\documentclass[12pt]{article}

\usepackage[T1]{fontenc}
\usepackage[utf8]{inputenc}
\usepackage[english]{babel}
\usepackage{fullpage}
\usepackage{graphicx}
\usepackage{subfig}
\usepackage{float} 
\usepackage{tikz}
\usetikzlibrary{patterns}
\usepackage{hyperref}

\usepackage{amsmath}
\usepackage{amsthm}
\usepackage{amssymb}
\usepackage{mathtools}

\newcommand{\al}{\alpha}
\newcommand{\Dt}{\Delta}
\newcommand{\dt}{\delta}
\newcommand{\ep}{\varepsilon}
\newcommand{\lb}{\lambda}

\newcommand{\Nent}{\mathbb{N}}
\newcommand\E[1]{\mathbb{E}\left[#1\right]}
\newcommand\pr[1]{\mathbb{P}\left[#1\right]}
\newcommand\prV[2]{\mathbb{P}_{\scriptscriptstyle #1}\left[#2\right]}

\newcommand{\tni}{\xrightarrow[n\to\infty]{}}
\newcommand{\sni}{\underset{n \to \infty}{\sim}}
\newcommand\Indi[1]{\textbf{1}_{\left\{#1\right\}}}

\newcommand{\cLaw}{\xrightarrow[n\to\infty]{d}}
\newcommand{\cPr}{\xrightarrow[n\to\infty]{\mathbb{P}}}
\newcommand\Binlaw[2]{\mbox{Bin}\left(#1,#2 \right)}
\newcommand\Explaw[1]{\mbox{Exponential}\left(#1\right)}
\newcommand\Weib[1]{\mbox{Weibull}\left(#1\right)}
\newcommand\Frec[1]{\mbox{Fr\'echet}\left(#1\right)}

\newcommand\dH[2]{d_{\mathbb{H}}(#1,#2)}
\newcommand{\RHG}{\mathcal{G}_{\al,\nu}(n)}
\newcommand{\RHGP}{\mathcal{G}^{\scriptscriptstyle{\text{Poi}}}_{\al,\nu}(n)}
\newcommand{\calP}{\mathcal{P}}
\newcommand{\ca}{C_{\alpha}}
\newcommand\Xo[1]{X_{(#1)}}
\newcommand{\Dmax}{D^{\mbox{\tiny{max}}}_n}

\newcommand\Br[2]{\mathcal{B}_{#1}(#2)}
\newcommand\munB[2]{\mu_n(\mathcal{B}_{#1}(#2))}
\newcommand{\Cr}{\mathcal{C}}
\newcommand\Brep[2]{\mathcal{B}^{\ep}_{#1}(#2)}
\newcommand\degep{\deg_{\ep}}
\newcommand\Brp[2]{\mathcal{B}'_{#1}(#2)}
\newcommand\degp{\deg'}
\newcommand\Brpp[2]{\mathcal{B}''_{#1}(#2)}
\newcommand\degpp{\deg''}
\newcommand\degtld{\widetilde{\deg}}
\newcommand\lbr[1]{\lb_n(r(\Xo{#1}))}

\DeclareMathOperator\arcosh{arcosh}

\newtheorem{thm}{Theorem}[section]
\newtheorem{prop}[thm]{Proposition}
\newtheorem{lem}[thm]{Lemma}
\newtheorem{rem}[thm]{Remark}

\title{Ordering and Convergence of Large Degrees\\ in Random Hyperbolic Graphs}
\author{Lo\"ic Gassmann\footnote{Universit\'e de Fribourg, Switzerland (loic.gassmann@unifr.ch)}
}
\date{\today}

\begin{document}

\maketitle

\begin{abstract}
We describe the asymptotic behaviour of large degrees in random hyperbolic graphs, for all values of the curvature parameter~$\al$. We prove that, with high probability, the node degrees satisfy the following ordering property: the ranking of the nodes by decreasing degree coincides with the ranking of the nodes by increasing distance to the centre, at least up to any constant rank. In the scale-free regime~$\al>1/2$, the rank at which these two rankings cease to coincide is~$n^{1/(1+8\al)+o(1)}$. We also provide a quantitative description of the large degrees by proving the convergence in distribution of the normalised degree process towards a Poisson point process. In particular, this establishes the convergence in distribution of the normalised maximum degree of the graph. A transition occurs at~$\al = 1/2$, which corresponds to the connectivity threshold of the model. For~$\al < 1/2$, the maximum degree is of order~$n - O(n^{\al + 1/2})$, whereas for~$\al \geq 1/2$, the maximum degree is of order~$n^{1/(2\al)}$. In the cases~$\al < 1/2$ and~$\al > 1/2$, the limit distribution of the maximum degree belongs to the class of extreme value distributions (Weibull for~$\al < 1/2$ and Fr\'echet for~$\al > 1/2$). This refines previous estimates on the maximum degree for~$\al > 1/2$ and extends the study of large degrees to the dense regime~$\al \leq 1/2$.
\end{abstract}

\noindent \textbf{MSC2020 subject classifications:} 05C80, 05C07, 60G70.

\section{Introduction}
The class of \emph{complex networks} consists of large real-life networks that primarily arise from human interactions, such as social networks and the Internet, as well as from other fields like biology \cite{AlbBar2002}. Networks in this class exhibit four essential features: \emph{high clustering}, the \emph{small-world} property, \emph{sparseness} and a \emph{scale-free degree distribution}~\cite{ChuLu2006}. Krioukov, Papadopoulos, Kitsak, Vahdat and Bogu\~{n}\'a empirically showed that these four properties naturally emerge in graphs constructed on hyperbolic spaces. This led them to introduce the \emph{random hyperbolic graph} (also denoted RHG) as a model for \emph{complex networks}~\cite{KriPapKitVahBog2010}. Bogu\~{n}\'a, Papadopoulos and  Krioukov further illustrated this point by providing an embedding of the Internet graph into a hyperbolic space~\cite{BogPapKri2010}. For modelling purposes, the model can be tuned through a curvature parameter~$\al$ and a parameter~$\nu$ determining the average degree.

It has now been rigorously proven that, in the regime~$\al > 1/2$, the \emph{random hyperbolic graph} exhibits all the properties of \emph{complex networks} listed above. \emph{Sparseness} is proven in~\cite{Pet2014}, the \emph{small-world} property is shown in~\cite{AbdBodFou2017}, the \emph{high clustering} is also fully tested \cite{CanFou2016,FouHooMulSche2021,GugPanPet2012} and the \emph{scale-free degree distribution} of the RHG has been proven in~\cite{GugPanPet2012}. In this regime, the \emph{random hyperbolic graph} is a particular case of the \emph{geometric inhomogeneous random graph} (also called GIRG)~\cite{BriKeuLen2017}. In the GIRG model, nodes are sampled on a torus and inhomogeneity is obtained via power law weights on the nodes, which determine the connection probabilities. It is established in~\cite{BriKeuLen2017,BriKeuLen2024} that the GIRG model possesses the main properties of \emph{complex graphs} listed above, thereby re-proving, in a more conceptual way, that the \emph{random hyperbolic graph} exhibits all the properties of \emph{complex networks}.

A graph is said to have a \emph{scale-free degree distribution} when its degree sequence follows a power law distribution, meaning that for large~$k$, the number of nodes with degree~$k$ behaves like an inverse power of~$k$. In this case, the graph has a large number of hubs (nodes with degrees much larger than the average degree). A \emph{random hyperbolic graph} is typically structured as follows: high-degree nodes are well connected to each other and serve as hubs for nodes with slightly lower degrees. These intermediate-degree nodes, in turn, connect to nodes with even lower degrees, and so on, forming a hierarchical, tree-like structure. In addition to this tree-like organization, there are also connections between nodes with similar angular coordinates. This structure provides an intuitive explanation for the \emph{small-world} phenomenon observed in these graphs, as short paths naturally emerge by travelling from hub to hub.

In the case~$\al \leq 1/2$, the \emph{random hyperbolic graph} is no longer \emph{sparse} nor \emph{scale-free}. However, the model remains highly inhomogeneous in the sense that the degree distribution of a node depends heavily on its position in the hyperbolic space, leading to the presence of a large number of hubs. This regime is referred to as the \emph{dense regime}. The value~$\al = 1/2$ is also the transition point between the connected and the non-connected regime: for~$\al < 1/2$, the graph has a high probability of being connected, with connectivity being entirely ensured by a few large hubs located near the centre. Conversely, for~$\al > 1/2$, the graph has a high probability of being disconnected. In the critical phase~$\al = 1/2$, the probability of connectivity tends to a constant that depends on the parameter~$\nu$. This constant takes the value~$1$ if and only if~$\nu \geq \pi$ (see~\cite{BodFouMul2016}).\medskip

\textbf{Main results}. In this paper, we are interested in the nodes with the largest degrees, which are the most important hubs of the graph. Theorem~\ref{thm:constant_rank_ordering} proves the ordering of these nodes, namely, that for~$k$ fixed, with high probability, the node with the~$k$-th largest degree is the node with the~$k$-th smallest distance to the centre of the underlying space. In the regime~$\al > 1/2$, Theorem~\ref{thm:kn_ordering} even shows that this ordering property holds up to rank~$n^{1/(1+8\al) + o(1)}$ and fails beyond. Finally, Theorem~\ref{thm:conv_distr_deg} states the convergence in distribution of the normalised point process of the degrees towards a Poisson point process. In particular, it establishes the convergence in distribution of the normalised maximum degree of the graph, for all $\alpha > 0$.

Node degrees in RHGs are closely related to the measures of certain regions of the underlying hyperbolic space. The exact expressions of these quantities are seldom tractable, but since we are seeking asymptotic results, we only need to approximate them. The approximations we employ depend on the value of the curvature parameter~$\al$, as the asymptotic position of the closest node to the centre is strongly influenced by this parameter. In the regime~$\al>1/2$, we use the approximations from~\cite[Lemma 3.2]{GugPanPet2012} (see Lemma~\ref{lem:approx_munB_al_large} of this paper and Lemma~\ref{lem:differential_munBr} for a refinement). In the two regimes~$\al<1/2$ and~$\al = 1/2$, we use new approximations (given by Lemmas~\ref{lem:approx_munB_al_small} and~\ref{lem:conv_munB_n}, respectively).\medskip

\textbf{Structure of the paper}. Section~\ref{section:preliminaries} contains a presentation of the random hyperbolic graph model. Our main results are presented in Section~\ref{section:results}. In Section~\ref{section:convergence_radii}, we prove the convergence of the node radii. Section~\ref{section:geo_lemmas} is dedicated to estimations of the measure of the balls involved in the connection rule. The results of these two last sections are used to prove Theorem~\ref{thm:constant_rank_ordering} (in Section~\ref{section:constant_rank_ordering}), Theorem~\ref{thm:conv_distr_deg} (in Section~\ref{section:conv_distr_deg}) and Theorem~\ref{thm:kn_ordering} (in Section~\ref{section:kn_ordering}).

\section{Definition of the model}
\label{section:preliminaries}

\subsection{Hyperbolic geometry}
Before introducing the random hyperbolic graph model, let us review some definitions and notations concerning hyperbolic geometry. We refer to the book of Stillwell~\cite{Sti1992} for a broader introduction to hyperbolic geometry. The Poincar\'e disc, denoted by~$\mathbb{H}$, is the open unit disc of~$\mathbb{C}$ equipped with the Riemannian metric~$\textbf{g}_\mathbb{H}$, defined at~$w \in \mathbb{H}$ by
\begin{align*}
\textbf{g}_\mathbb{H} \coloneqq \frac{4\textbf{g}_{\mathbb{C}}}{(1-|w|^2)^2}, \quad \mbox{where~$\textbf{g}_{\mathbb{C}}$ is the Euclidean metric on~$\mathbb{C}$.}
\end{align*}
We denote by~$d_{\mathbb{H}}$ the distance induced by~$\textbf{g}_{\mathbb{H}}$ on~$\mathbb{H}$. Throughout this paper, we make extensive use of the polar coordinates to describe points in the Poincar\'e disc. The polar coordinates of a points~$w$ in~$\mathbb{H}$ are given by the pair~$(r(w),\theta(w))$, where~$r(w)$ denotes its hyperbolic distance to the origin and~$\theta(w)$ denotes its angle in the complex plane. The quantity~$r(w)$ is also referred to as the radius of~$w$. If~$x$ and~$y$ are two points of~$\mathbb{H}$ with respective polar coordinates~$(r,\theta)$ and~$(s,\beta)$, the hyperbolic distance between~$x$ and~$y$ is given by the celebrated hyperbolic law of cosines:
\begin{align}\label{formula:hyperbolic_cos}
\cosh(\dH{x}{y}) = \cosh(r)\cosh(s) - \sinh(r)\sinh(s)\cos(\theta - \beta).
\end{align}

For visualization purposes, we will use the native representation to draw pictures of random hyperbolic graphs, as done in~\cite{KriPapKitVahBog2010}. This means that instead of representing random hyperbolic graphs directly in~$\mathbb{H}$, we will represent their images under the mapping~$\omega \mapsto r(w)e^{i\theta(w)}$, defined from~$\mathbb{H}$ to~$\mathbb{C}$. This transformation dilates all distances to the origin, ensuring that every point~$(r,\theta)$ is represented with a Euclidean distance to the origin equal to its radial coordinate~$r$.

For a point~$x$ with polar coordinates~$(r,\theta)$ in~$\mathbb{H}$ and a radius~$s > 0$, we denote by~$\Br{x}{s}$ or~$\Br{(r,\theta)}{s}$ the open hyperbolic ball of radius~$s$ centred at~$x$. For~$0 < r_1 < r_2$, we define~$\Cr(r_1,r_2)$ as the annulus with inner radius~$r_1$ and outer radius~$r_2$, i.e.,~$\Cr(r_1,r_2)\coloneqq \Br{0}{r_2} \setminus \Br{0}{r_1}$.

\subsection{The Random Hyperbolic Graph}
\label{subsection:def_RHG}

We now proceed with the formal definition of the random hyperbolic graph~$\RHG$, as defined in~\cite{KriPapKitVahBog2010}. Fix two parameters~$\al > 0$ and~$\nu > 0$ and for~$n \in \Nent^*$, set
\begin{align}
R_n \coloneqq 2\log(n/\nu).
\end{align}
Define a probability measure~$\mu_n$ on~$\mathbb{H}$ such that if a point~$(r,\theta)$ (in polar coordinates) is chosen according to~$\mu_n$, then~$r$ and~$\theta$ are independent,~$\theta$ is uniformly distributed in~$(-\pi,\pi]$ and the probability distribution of~$r$ has a density function on~$\mathbb{R}_+$ given by
\begin{align}\label{def:density}
\rho_n(r) \coloneqq \frac{\al \sinh(\al r)}{(\cosh(\al R_n) - 1)} \Indi{r < R_n}.
\end{align}
Let~$X_1^n,X_2^n,\dots,X_n^n$ be a sequence of~$n$ independent points sampled from the Poincar\'e disc according to the distribution~$\mu_n$ (for brevity, the superscript~$n$ in~$X_i^n$ will often be omitted). We denote by~$\RHG$ the random hyperbolic graph with~$n$ nodes and parameters~$\al$ and~$\nu$. It is defined as the undirected graph with nodes at the points~$X_1^n,X_2^n,\dots,X_n^n$, where an edge exists between two nodes if and only if their hyperbolic distance is at most~$R_n$.

The degree of a node~$X_i$ in the graph~$\RHG$ is defined as the number of its direct neighbours in~$\RHG$ and is denoted by~$\deg(X_i)$. Since we focus on the behaviour of large graphs, the value of~$n$ will always be considered large, while the parameters~$\al$ and~$\nu$ are fixed. We say that an event is realised with high probability if its probability tends to~$1$ as~$n$ goes to infinity.

\begin{figure}
  \centering
  \subfloat{\includegraphics[scale=0.28]{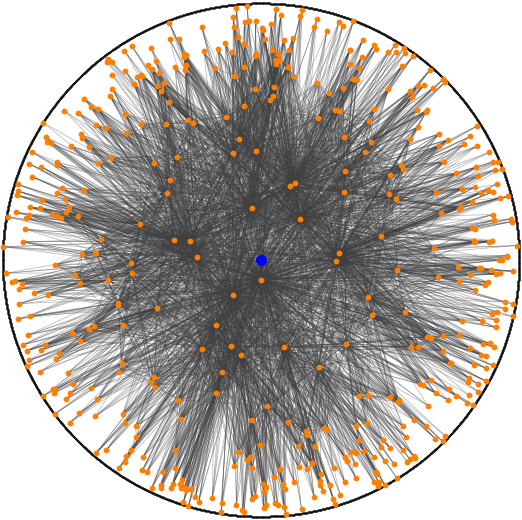}}
  \hspace{12mm}
  \subfloat{\includegraphics[scale=0.28]{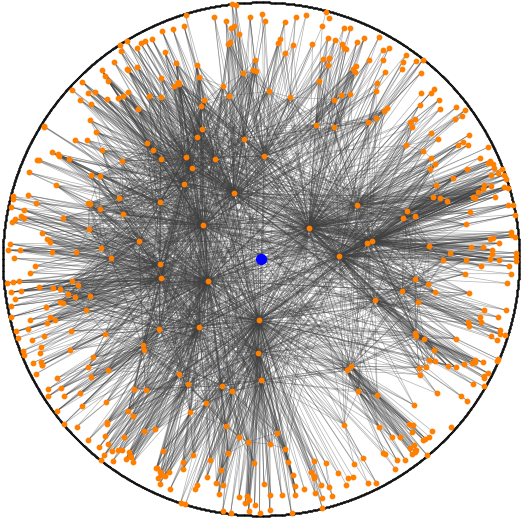}}
  \hspace{12mm}
  \subfloat{\includegraphics[scale=0.28]{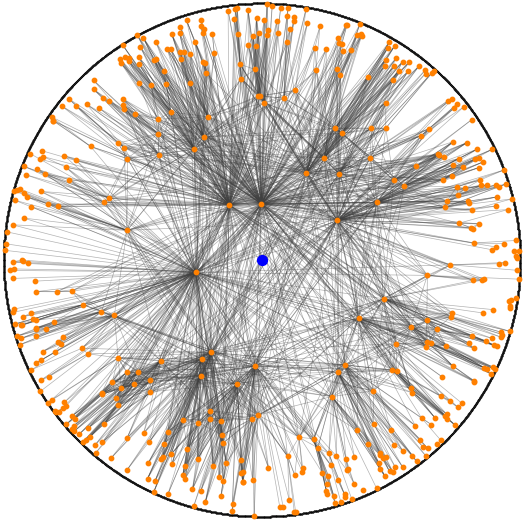}}
  \caption{Simulations of random hyperbolic graphs (native representation) with~$n = 500$, $\nu = 1$,~$\al = 0.45$ (left),~$\al = 0.50$ (middle) and~$\al = 0.55$ (right). The boundary of~$\Br{0}{R_n}$ is represented by a black circle and its centre by a blue dot.}\label{fig:simulations}
\end{figure}

Observe that the~$n$ nodes of the graph~$\RHG$ are located within~$\Br{0}{R_n}$. Moreover, due to the choice of the measure, the points tend to concentrate near the boundary of~$\Br{0}{R_n}$. Also note that the measure~$\munB{(r,\theta)}{s}$ of a ball centred at~$(r,\theta)$ is independent of~$\theta$. Therefore, we omit the angle~$\theta$ to shorten notation and instead write~$\munB{r}{s}$. Likewise, we write~$\mu_n(\Br{r}{s} \setminus \Br{0}{s'})$ instead of~$\mu_n(\Br{(r,\theta)}{s} \setminus \Br{0}{s'})$.

In the special case~$\al = 1$,~$\mu_n$ is the uniform measure on~$\Br{0}{R_n}$ associated with the Riemannian metric~$\textbf{g}_\mathbb{H}$. In the general case~$\al > 0$,~$\mu_n$ corresponds to a uniform measure on the hyperbolic plane~$\mathbb{H}_\al$ of curvature~$-\al^2$. More precisely, for fixed~$\al > 0$, multiplying the differential form in the Poincar\'e disc model by a factor~$1/\al^2$ we obtain a hyperbolic plane of curvature~$-\al^2$. Choosing a point according to the measure~$\mu_n$ amounts to choosing a point uniformly in the ball of radius~$R_n$ of~$\mathbb{H}_\al$ and projecting it on~$\mathbb{H}$, by keeping the same polar coordinates. The measure of a ball of radius~$r$ in~$\mathbb{H}_\al$ is 
\begin{align*}
\frac{2\pi}{\al^2}(\cosh(\al r) - 1),
\end{align*}
thus, the larger~$\al$ is, the faster it increases with~$r$. Therefore, the larger~$\al$ is, the more the points of the graph~$\RHG$ concentrate near the boundary of~$\Br{0}{R_n}$. So, the maximum degree is expected to decrease with~$\al$ (see Figure~\ref{fig:simulations}). The degree distribution is also influenced by the parameter~$\nu$. Increasing~$\nu$ makes the domain smaller, which limits the concentration of the nodes near the boundary of the domain, resulting in a higher expected degree. In the scale-free regime~$\al > 1/2$, the expected degree evolves linearly with~$\nu$~\cite{GugPanPet2012}. 

The Poissonised random hyperbolic graph~$\RHGP$ is obtained by choosing the nodes according to a Poisson point process with intensity measure~$n\mu_n$, instead of choosing~$n$ points according to the measure~$\mu_n$. All the results proved in this paper also hold for the Poissonised model. In the proof of the non-ordering result~\eqref{assert:no_ordering_after_n_beta}, the use of a classical Poissonisation/de-Poissonisation procedure allows to avoid some technicalities by using the properties of Poisson processes.

\section{Results} \label{section:results}

\subsection{Ordering of Large Degrees}

We denote by~$\Xo{1}^n,\Xo{2}^n,\dots,\Xo{n}^n$ a reordering of the nodes of~$\RHG$ by increasing radius, i.e.,~$r(\Xo{1}^n) \leq r(\Xo{2}^n) \leq \dots \leq r(\Xo{n}^n)$. To shorten notation, we will often omit the superscript~$n$. Our first result shows that this ranking of the nodes coincide with the ranking of the nodes by decreasing degree, at least up to any constant rank.

\begin{thm}\label{thm:constant_rank_ordering}
For fixed~$\al > 0$,~$\nu > 0$ and~$k \in \Nent^*$, with high probability,
\begin{align*}
\forall i > k, \; \deg(\Xo{1}^n) > \deg(\Xo{2}^n) > \dots > \deg(\Xo{k}^n) > \deg(\Xo{i}^n)
\end{align*}
\end{thm}

In the regime~$\al > 1/2$, we will even prove the following result, which provides an estimate of the rank at which the ranking of the nodes by increasing radius and the ranking of the nodes by decreasing degree cease to coincide. We believe that similar polynomial rank orderings hold in the other regimes, but we choose to present this refined result only for the case~$\alpha > 1/2$ to avoid additional computations.
\begin{thm}\label{thm:kn_ordering}
Let us fix~$\alpha > 1/2$, $\nu > 0$ and a sequence~$v_n \to \infty$. Define
\begin{align*}
\beta \coloneqq \frac{1}{1+8\al} \quad \mbox{and} \quad k_n &\coloneqq n^{\beta}  \log(n)^{-2\al}.
\end{align*}
We have, with high probability,
\begin{align}\label{assert:ordering_up_to_k_n}
\forall i>k_n, \; \deg(\Xo{1}^n) > \deg(\Xo{2}^n) &> \dots > \deg(\Xo{k_n}^n) > \deg(\Xo{i}^n) 
\end{align}
and there exists~$i \in [n^{\beta},n^{\beta}v_n]$ such that
\begin{align}\label{assert:no_ordering_after_n_beta}
\deg(\Xo{i}^n) < \deg(\Xo{i+1}^n).
\end{align}
\end{thm}
If we choose, for example,~$v_n = \log(n)$, then the result shows that, with high probability, the ordering of the degree holds up to rank~$n^{\beta-o(1)}$ and fails before rank~$n^{\beta + o(1)}$. This proves the optimality of the exponent~$\beta$ as the ordering exponent.

The ordering results above show that the position of a node in the underlying hyperbolic space is a precise estimate of its position in the hierarchy of hubs. This is natural, as the degree of the~$k$ closest nodes to the centre stochastically dominates the degree of the~$n-k$ following nodes. However, this observation alone is insufficient to conclude, since having more connections than these~$n-k$ nodes requires competing with a polynomial number of nodes. For example, in the random recursive tree model, similar ordering results do not hold, as the root has the highest expected degree but does not necessarily have the highest degree (see~\cite{DevLu1995}). In RHGs, we use the fast decay of degree with distance to the centre to show that the~$k$ (or~$k_n$) closest nodes to the centre can compete with the remaining~$n-k$ (or~$n-k_n$) nodes. 

\subsection{Convergence of Large Degrees}

For every subspace~$E$ of the compactified real line~$[-\infty,\infty]$ we denote by~$M_p(E)$ the space of locally finite point measures on~$E$. A point process on~$E$ is a random element of~$M_p(E)$ (see~\cite{Res1987,Bre2020,Kal2017} for details on point processes). We denote by~$\xrightarrow{d}$ the convergence in distribution (or weak convergence) of point processes in~$M_p(E)$. We will also use the notation~$\xrightarrow{d}$ for the convergence in distribution of random variables or random vectors. For~$x \in [-\infty,\infty]$, we write~$\dt_x$ for the Dirac measure at~$x$.

We want to describe the asymptotic behaviour of~$\sum_{i=1}^n \dt_{\deg(X_i^n)}$, which is the point process of node degrees . We consider this point process as an element of~$M_p((0,\infty])$. Since the interval includes~$+\infty$, this point process captures information about the largest degrees of the graph. Given that the expected value of the maximum degree goes to infinity with~$n$, we normalise the degree process by a quantity depending on the expected value of the maximum degree. This quantity heavily depends on the curvature parameter~$\al$. As with connectivity, a transition occurs at~$\al = 1/2$: for~$\al < 1/2$, the maximum degree is of order~$n - O(n^{\al + 1/2})$, whereas for~$\al \geq 1/2$, the maximum degree is of order~$n^{1/(2\al)}$. This compels us to treat the three regimes~$\al<1/2$,~$\al=1/2$ and~$\al > 1/2$ separately. The result below states the convergence in distribution of the normalised point process. We recall that for~$\beta > 0$, a random variable~$Z$ follows the distribution~$\Weib{\beta}$ if for all~$z \geq 0$,~$\mathbb{P}[Z \leq z] = 1 - e^{-z^\beta}$ and it follows the distribution~$\Frec{\beta}$ if for all~$z \geq 0$,~$\mathbb{P}[Z \leq z] =  e^{-z^{-\beta}}$.

\begin{thm}\label{thm:conv_distr_deg} Fix~$\al,\nu > 0$ and denote by~$\Dmax$ the maximum degree of the graph~$\RHG$. We have the following convergences:\\
$\bullet$ For~$\alpha < 1/2$, 
\begin{align*}
\sum\limits_{i=1}^n \dt_{n^{-(\al + 1/2)}(\deg(X_i^n)-n)} \cLaw \eta_{g_1},\quad \mbox{in } M_p((-\infty,0]),
\end{align*}
where~$\eta_{g_1}$ is the Poisson process whose intensity measure has density~$g_1$, given by
\begin{align*}
\forall y \in (-\infty,0],\, g_1(y) = 2 \pi^2 \nu^{2\al} |y|.
\end{align*}
In particular, 
\begin{align*}
\pi\nu^{\al} \frac{n-\Dmax}{n^{\al+1/2}} \cLaw \Weib{2}.
\end{align*}
$\bullet$ For~$\alpha = 1/2$,
\begin{align*}
\sum\limits_{i=1}^n \dt_{n^{-1}\deg(X_i^n)} \cLaw \eta_{g_2},\quad \mbox{in } M_p((0,1]),
\end{align*}
where~$\eta_{g_2}$ is the Poisson process whose intensity measure has density~$g_2$, given by
\begin{align*}
\forall y \in (0,1],\, g_2(y) = \nu |(V_{\scriptscriptstyle 1/2}^{-1})'(y)| \sinh(V_{\scriptscriptstyle 1/2}^{-1}(y)/2),
\end{align*}
where~$V_{\scriptscriptstyle 1/2}$ is defined by
\begin{align*}
\forall r > 0,\; V_{\scriptscriptstyle 1/2}(r) \coloneqq \frac{1}{\pi}\int_0^1 \arccos\left(\max\left(-1,\frac{\cosh(r) - x^{-2}}{\sinh(r)}\right)\right) dx.
\end{align*}
In particular,
\begin{align*}
n^{-1} \Dmax \cLaw V_{\scriptscriptstyle 1/2}(2\arcosh(\Explaw{2\nu}+1)).
\end{align*}
$\bullet$ For~$\alpha > 1/2$,
\begin{align*}
\sum\limits_{i=1}^n \dt_{n^{-1/(2\al)}\deg(X_i^n)} \cLaw \eta_{g_3},\quad \mbox{in } M_p((0,\infty]),
\end{align*}
where~$\eta_{g_3}$ is the Poisson process whose intensity measure has density~$g_3$, given by
\begin{align*}
\forall y \in (0,\infty],\, g_3(y) = 2\al (\ca \nu)^{2\alpha} y^{-2\al-1},
\end{align*}
where~$\ca \coloneqq \frac{2\al}{\pi(\al-1/2)}$. In particular,
\begin{align*}
\frac{\Dmax}{\ca \nu n^{1/(2\al)}} \cLaw \Frec{2\al}.
\end{align*}
\end{thm}

\begin{rem}\label{rem:thm_conv}
Convergence in distribution of a sequence of processes~$\calP_n$ towards a process~$\calP$ in~$M_p((a,b])$ (with~$a,b \in [-\infty,\infty]$) is equivalent to the convergence in distribution of the vectors~$(Y_1^n,Y_2^n,\dots,Y_k^n)$ towards~$(Y_1,Y_2,\dots,Y_k)$ in~$\mathbb{R}^k$ for every positive integer~$k$, where~$Y_1^n \geq Y_2^n \geq \dots \geq Y_k^n$ are the~$k$ largest points of~$\calP_n$ and~$Y_1 \geq Y_2 \geq \dots \geq Y_k$ are the~$k$ largest points of~$\calP$ (assuming that~$\calP$ almost surely has infinitely many points). Thus, the convergence of the maximum degree follows directly from the convergence of the point process of the degrees. One get a similar characterisation of convergence in distribution in~$M_p([a,b))$ by inverting the order. Convergence in distribution in~$M_p(E)$ can be characterised in many other ways, such as via the convergence of evaluation functions or the convergence of Laplace transforms (see~\cite{Res1987,Bre2020,Kal2017}).
\end{rem}

In the regime~$\al > 1/2$, Theorem \ref{thm:conv_distr_deg} refines the estimate of~$n^{1/(2\al)+o(1)}$ for the maximum degree given in~\cite{GugPanPet2012}. It also extends the description of the scale-free degree sequence presented in the same paper by describing precisely the sequence of the~$k$ largest degrees for~$k$ fixed. Also note that in the regime~$\al > 1/2$, a more general result can be found in~\cite{BhaSch2022}, where the convergence towards a Poisson point process with a power law intensity is proven for a general class of scale-free inhomogeneous random graphs. By refining the proof of the representation of the RHG as a GIRG from~\cite{BriKeuLen2017}, one can show that the random hyperbolic graph in the regime~$\al > 1/2$ is indeed part of the more general model considered in~\cite{BhaSch2022}. This provides an alternative proof of Theorem~\ref{thm:conv_distr_deg} for~$\al > 1/2$. Note that this result would not hold directly for GIRGs, even for unspecified limit distributions, because the~$O$-notations appearing in the definition of the edge probabilities may allow large degrees to oscillate instead of converging. In the present paper, we propose proof strategies that work for both the scale-free regime~$\al > 1/2$ and the dense regime~$\al \leq 1/2$.

In the cases~$\al < 1/2$ and~$\al > 1/2$, we obtain extreme value distribution limits for the maximum degree~$\Dmax$. This is not surprising, as~$\Dmax$ is the maximum of a weakly correlated sequence of variables. However, a difficulty in proving a formal result using a standard extreme value theorem (see for example~\cite{LeaRoo1988,Res1987}) arises from the fact that the distribution of the typical degree depends on~$n$. Nevertheless, we can still use extreme value theory to give an intuitive explanation of the limit distributions of~$\Dmax$ as follows: In the regime~$\al > 1/2$, the limit distribution is easy to interpret since the degree of a typical node converges to a power law with exponent~$2\al +1$, which belongs to the domain of attraction of~$\Frec{2\al}$. In the regime~$\al < 1/2$, the degree distribution does not converge, but since the largest degrees are attained by nodes that are close to the centre, we can focus on a variable~$D_n$ distributed as the degree of a typical node in the annulus~$\Cr(0,n^{\al-1/2 + \ep})$, with~$\ep > 0$ small. A direct computation using Lemma~\ref{lem:approx_munB_al_small} shows that~$(n-D_n)n^{-(1/2+\al+\ep)}$ converges in distribution towards the inverse of a power law with exponent~$1$, which belongs to the domain of attraction of~$\Weib{2}$. Similar arguments using polynomial normalisations cannot hold in the regime~$\al = 1/2$ as, in this regime, the highest degrees are too much influenced by the geometry of the model. The influence of the geometry can be seen in the function~$V_{\scriptscriptstyle 1/2}(2\arcosh(\cdot+1))$ appearing in the limit and explains why we do not obtain an extreme value distribution in this case.

In a soft version of the random hyperbolic graph model~\cite{BogPapKri2010}, the threshold rule for connection is replaced by the following connection rule: conditionally on their positions, the two nodes~$X_i$ and~$X_j$ are connected with probability
\begin{align*}
p_{ij} = \frac{1}{1 + \exp((d_{\mathbb{H}}(X_i,X_j) - R_n)/(2T))},
\end{align*}
where~$T>0$ is a temperature parameter. We recover our initial model by taking the limit~$T \to 0$ in the soft model. We believe that the results of Theorems~\ref{thm:constant_rank_ordering} and~\ref{thm:conv_distr_deg} remain valid in the soft model (with different limit distributions in Theorem \ref{thm:conv_distr_deg}), as the expected degrees of the nodes should stay of the same order in the soft model. In the case~$\al > 1/2$, this can be proven by combining~\cite{BhaSch2022} and~\cite{BriKeuLen2017}. The computation of the limit measure in the soft case would require replacing volume estimates with estimates of the integrals of the connection probability, in order to approximate the expected degrees of the nodes under the new connection rule. \medskip

\textbf{Landau's notations:} In this paper, we use the Landau notations~$o$ and~$O$ to describe the asymptotic behaviour of certain quantities as the number of nodes~$n$ tends to infinity. Specifically, we use a version of these notations that allows us to express uniformity in some other variables~$y_1,y_2,...,y_d$. More precisely, for~$(J_n)_{n \in \Nent}$ a sequence of subsets of~$\mathbb{R}^d$ and two functions~$f,g:\Nent \times \mathbb{R}^d \to \mathbb{R}$, we write
\begin{align*}
\quad g(n,y) = o_{y \in J_n}(f(n,y)) \quad \mbox{if} \quad \forall \ep > 0, \exists n_0, \forall n \geq n_0, \forall y \in J_n,\; |g(n,y)| \leq \ep |f(n,y)|,
\end{align*}
and we write
\begin{align*}
\quad g(n,y) = O_{y \in J_n}(f(n,y)) \quad \mbox{if} \quad \exists c > 0, \exists n_0, \forall n \geq n_0, \forall y \in J_n,\; |g(n,y)| \leq c |f(n,y)|.
\end{align*}
Note that the subscript "$y \in J_n$" indicates that the comparison holds uniformly for~$y$ in the set~$J_n$ as~$n$ tends to infinity. To simplify notation, this subscript is specified beforehand and omitted in most cases. When~$f$ and~$g$ are functions of the variable~$n$ only, we keep the same definitions.

\section{Convergence of the Radii \label{section:convergence_radii}}
The particularity of random hyperbolic graphs, compared to graphs constructed in Euclidean spaces, is that the degree of a node strongly depends on its position in the underlying space, specifically on its radius. In this section, we study the asymptotic behaviour of the radii of the closest points to the centre by proving the convergence of the point process of the radii~$\sum_{i=1}^n \dt_{r(X_i^n)}$ in~$M_p([0,\infty))$. Since the interval includes~$0$, this point process captures information about the smallest degrees of the graph (see Remark~\ref{rem:thm_conv}). As in Theorem~\ref{thm:constant_rank_ordering}, normalisation by the expected value of the smallest radius is required.
\begin{prop}\label{prop:convergence_of_curveR}
We have the following convergences:\\
$\bullet$ For~$\alpha < 1/2$, 
\begin{align*}
\sum\limits_{i=1}^n \dt_{n^{1/2-\al}r(X_i^n)} \cLaw \eta_{\gamma_1}, \quad \mbox{in } M_p([0,\infty)),
\end{align*}
where~$\eta_{\gamma_1}$ is the Poisson process whose intensity measure has density~$\gamma_1$, given by
\begin{align*}
\forall  u \in [0,\infty) ,&\; \gamma_1(u) = 2 \al^2 \nu^{2\al} u.
\end{align*}
$\bullet$ For~$\alpha = 1/2$,
\begin{align*}
\sum\limits_{i=1}^n \dt_{r(X_i^n)} \cLaw \eta_{\gamma_2} ,\quad \mbox{in } M_p([0,\infty)),
\end{align*}
where~$\eta_{\gamma_2}$ is the Poisson process whose intensity measure has density~$\gamma_2$, given by
\begin{align*}
\forall u \in [0, \infty),&\; \gamma_2(u) = \nu \sinh(u/2).
\end{align*}
$\bullet$ For~$\alpha > 1/2$,
\begin{align*}
\sum\limits_{i=1}^n \dt_{r(X_i^n) - (1-1/(2\al))R_n} \cLaw \eta_{\gamma_3},\quad \mbox{in } M_p([-\infty,\infty)),
\end{align*}
where~$\eta_{\gamma_3}$ is the Poisson process whose intensity measure has density~$\gamma_3$, given by
\begin{align*}
\forall u \in [-\infty, \infty) ,&\; \gamma_3(u) = \al \nu e^{\al u}.
\end{align*}
\end{prop}

\begin{proof}
We introduce the function~$\gamma_1$,~$\gamma_2$ and~$\gamma_3$ exactly like in the statement of the Proposition. For~$E$ a subspace of~$[-\infty,\infty]$, let us write~$C_K^+(E)$ for the set of continuous, real-valued and non-negative functions on~$E$ with compact support. We begin with the case~$\al < 1/2$. Fix~$\varphi \in C_K^+([0,\infty))$. By a simple change of variables, we get
\begin{align*}
n \E{\varphi(n^{1/2-\alpha}r(X_1^n))}
&= n \int_0^\infty \varphi(n^{1/2-\alpha}r) \rho_n(r) dr\\
&= \int_0^\infty \varphi(u) \frac{\al \sinh(\al u n^{\al-1/2})}{\cosh(\al R_n) - 1} n^{1/2+\al} \Indi{u \leq n^{1/2-\al} R_n} du.
\end{align*}
Using that~$\sinh(x) \sim x$ for~$x$ close to~$0$ and approximating~$\cosh(R_n/2)$ by~$n/(2\nu)$, we get by the dominated convergence theorem,
\begin{align}\label{limit:condition_for_lap_conv}
n \E{\varphi(n^{1/2-\alpha}r(X_1^n))} \tni \int_0^\infty \varphi(u) \gamma_1(u) du.
\end{align}
For all~$n \in \Nent^*$, we denote by~$\Psi_n$ the Laplace functional associated with the point process~$\sum_{i=1}^n \dt_{n^{1/2-\al}r(X_i^n)}$ and we denote by~$\Psi$ the Laplace functional associated with~$\eta_{\gamma_1}$. Since the variables~$r(X_1^n),\dots,r(X_n^n)$ are independent, it follows from~\eqref{limit:condition_for_lap_conv} and a classical computation that for all~$f \in C_K^+([0,\infty))$,~$\Psi_n(f) \to \Psi(f)$ (see~\cite[Proposition 3.21]{Res1987} for a detailed computation). This is sufficient to establish the desired convergence in distribution for $\al < 1/2$.

In the case~$\alpha = 1/2$, for~$\varphi \in C_K^+([0,\infty))$, we have
\begin{align*}
n \E{\varphi(r(X_1^n))} 
&= n \int_0^\infty \varphi(r) \frac{\sinh(r/2)}{2(\cosh(R_n/2) - 1)}\Indi{r \leq R_n}dr\\
&\tni  \int_0^\infty \varphi(r) \gamma_2(r) dr,
\end{align*}
where the convergence follows from approximating~$\cosh(R_n/2)$ by~$n/(2\nu)$ and applying the dominated convergence theorem. Thus, we can conclude exactly like in the case~$\alpha < 1/2$.

In the case~$\alpha > 1/2$, a simple change of variables gives, for~$\varphi \in C_K^+([-\infty,\infty))$,
\begin{align}\label{formula:lap_conv_large_al}
n \E{\varphi(r(X_1^n) - (1-\tfrac{1}{2\al})R_n)}
&= n \int_0^\infty  \varphi(r - (1-\tfrac{1}{2\al})R_n) \rho_n(r) dr \notag\\
&=  \int_{-(1-\tfrac{1}{2\al}) R_n}^{R_n / (2\al)}  \varphi(u) \frac{\al n \sinh(\al (u + (1-\tfrac{1}{2\al})R_n))}{\cosh(\al R_n) - 1} du.
\end{align}
Moreover, there exists~$K > 0$ such that for all sufficiently large~$n \in \Nent^*$ and all~$u \in \mathbb{R}$,
\begin{align*}
\varphi(u) \frac{\al n \sinh(\al (u + (1-\tfrac{1}{2\al})R_n))}{\cosh(\al R_n) - 1} \leq K \varphi(u) \exp(\al u).
\end{align*}
The bound above is integrable on~$\mathbb{R}$. Applying the dominated convergence theorem to~\eqref{formula:lap_conv_large_al} and using that~$\cosh(x) \sim e^{x}/2$ and~$\sinh(x) \sim e^{x}/2$ for~$x$ large, we get
\begin{align*}
n \E{\varphi(r(X_1^n) - (1-\tfrac{1}{2\al})R_n)} \tni \int_{-\infty}^{\infty} \varphi(u) \gamma_3(u) du.
\end{align*}
We conclude exactly like in the case~$\alpha < 1/2$.
\end{proof}

\section{Measures of the Connection Balls} \label{section:geo_lemmas}
A node in the random hyperbolic graph~$\RHG$ is connected to all other nodes that are at a distance smaller than~$R_n$ from it. Thus, conditionally on the position of the node~$X$, the degree of~$X$ follows a binomial distribution with~$n-1$ trials and success probability~$\munB{r(X)}{R_n}$. The previous section gives a good understanding of the radii of the closest nodes to the centre~$r(\Xo{1})$. In order to obtain information about the degrees of the closest nodes to the centre, we now study ~$\munB{r}{R_n}$ as a function of~$r$. Unfortunately, this quantity cannot be expressed by a tractable closed-form formula. However, since we are principally concerned with the asymptotic behaviour of~$\RHG$ as $n \to \infty$, it is sufficient to provide approximations of this quantity.

\newcommand{\C}{(0,0) circle (1.5)}
\newcommand{\rect}{(-2,-1) rectangle (2,1.4)}
\begin{figure}
\centering
\begin{tikzpicture}[scale = 1.5]
\draw \C;
\begin{scope}
\clip \rect;
\node[inner sep=0pt,opacity=0.4]  at (0,0) {\includegraphics[scale = 0.3] {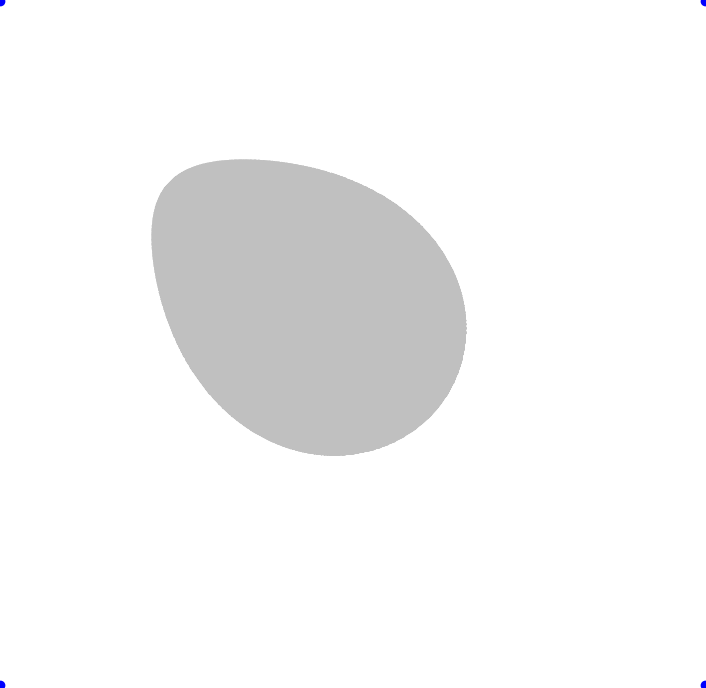}};
\end{scope}
\foreach \theta/\r in {10/1.15,275/0.55,40/0.75,130/1.3,270/1.2,20/1.2,210/0.8,52/1.45,-40/1.35,
-30/1.1,80/1.3,75/1,210/1.2,225/1.1,300/1.3,350/0.7}
{
    \draw (\theta:\r) node[scale = 0.7] {$\bullet$};          
}
\foreach \theta/\r/\alpha/\s in {138/0.41/130/1.3, 138/0.41/210/0.8, 138/0.41/40/0.75, 138/0.41/75/1, 138/0.41/275/0.55, 138/0.41/350/0.7, 210/0.8/210/1.2, 210/0.8/225/1.1, 40/0.75/52/1.45, 40/0.75/75/1, 40/0.75/20/1.2, 75/1/80/1.3, 20/1.2/10/1.15,275/0.55/300/1.3, 275/0.55/270/1.2,-40/1.35/-30/1.1 }
{
    \draw (\theta:\r) -- (\alpha:\s);          
}
\draw (0,0) node[below right,scale=0.7]{$0$} node[scale=0.7] {$\bullet$} ;
\draw(138:0.41) node[scale = 1] {$\bullet$} node [left] {$X$};
\draw  (20:1.5) node[right,scale = 0.7]  {$B_0(R_n)$};
\draw  (150:1.1) node[scale = 0.7] {$\Br{X}{R_n}$};
\end{tikzpicture}
\caption{Depiction of a ball~$\Br{X}{R_n}$ (native representation) \label{fig:image_bin_law}}
\end{figure}

Before dealing with quantitative results, we state the following lemma, which gives an intuitive inclusion between hyperbolic balls. Its proof follows readily from~\cite[Lemma 2.3]{BodFouMul2015}. The decreasing of~$\munB{r}{R_n}$ in~$r$ shows that the expected degree of a node~$X$ is a decreasing function of~$r(X)$. Indeed, one can see in Figure~\ref{fig:image_bin_law} that the ball~$\Br{X}{R_n}$ is not entirely contained in the domain~$\Br{0}{R_n}$. If we let~$r(X)$ increase, then more mass is lost outside the domain~$\Br{0}{R_n}$, leading to a smaller expected degree. This strongly suggests that the nodes with the highest degrees are likely to be located near the origin of the underlying space.
\begin{lem} \label{lem:decreasing_munB} For all~$\al,\nu > 0$,~$n \in \Nent^*$ and~$0\leq r \leq r' \leq R_n$, the following holds
\begin{align*}
\Br{(r',0)}{R_n} \cap \Br{0}{R_n} \; \subset \; \Br{(r,0)}{R_n} \cap \Br{0}{R_n}.
\end{align*}
In particular, the function~$r \mapsto \munB{r}{R_n}$ is decreasing.
\end{lem}

To get quantitative results on~$\munB{r(X)}{R_n}$, we first provide an integral expression of~$\munB{r}{R_n}$, for $r \in [0,R_n)$. In the following definitions, we take~$\al >0$, and fix an integer~$n \in \Nent^*$. We also take~$r$ and~$y$ in~$[0,R_n)$. Using polar coordinates for points in the Poincar\'e disc, we define the set~$I_y(r)$ by
\begin{align*}
I_y(r) \coloneqq \{\theta \in [0,\pi] : (y,\theta) \in \Br{(r,0)}{R_n}\}.
\end{align*}
A direct use of the hyperbolic law of cosines (\ref{formula:hyperbolic_cos}) yields that~$I_y(r)$ is a non-empty interval containing~$0$. So, defining~$\theta_r(y)$ by
\begin{align*}
\theta_r(y) \coloneqq \sup I_y(r),
\end{align*} 
we get
\begin{align}\label{formula:munB}
\munB{r}{R_n} = \frac{1}{\pi}\int_0^{R_n}\theta_r(y) \rho_n(y) dy.
\end{align}
Figure~\ref{fig:theta_y_r} below provides a graphical representation of the angle~$\theta_r(y)$ (in a Euclidean setting to ease the representation). Let us fix~$r \in [0,R_n)$. We want to compute~$\theta_r(y)$ for~$y \in [0,R_n)$. It is clear that~$\Br{0}{R_n-r} \subset \Br{(r,0)}{R_n}$, so
\begin{align}\label{formula:theta_small_y}
\forall y \in [0,R_n - r],\; \theta_r(y) = \pi.
\end{align}
When~$y \geq R_n - r$, the point~$(y,\theta_r(y))$ is at a distance~$R_n$ from the point~$(r,0)$, so by the hyperbolic law of cosines (\ref{formula:hyperbolic_cos}), we get
\begin{align}\label{formula:theta_big_y}
\forall y \in [R_n - r,R_n),\; \cos(\theta_r(y)) = \frac{\cosh(r)\cosh(y) - \cosh(R_n)}{\sinh(r) \sinh(y)}.
\end{align}
One may check that equations (\ref{formula:theta_small_y}) and (\ref{formula:theta_big_y}) can be rewritten as:
\begin{align}\label{formula:theta_gen}
\forall r,y \in [0,R_n),\; \theta_r(y) = \arccos\left(\max\left(-1,\frac{\cosh(r)\cosh(y) - \cosh(R_n)}{\sinh(r) \sinh(y)}\right)\right).
\end{align}

\begin{figure}
\begin{center}
\begin{tikzpicture}[scale = 0.7]
\draw  (0,0) circle (5);
\draw [fill=black!25, opacity=0.6] (2,0) circle (5);

\draw (160:5) node[above left] {$\Br{0}{R_n}$};
\draw (20:5)+(2,0) node[above right] {$\Br{(r,0)}{R_n}$};

\draw [dashed] (0,0) -- (5,0);
\draw [dashed] (2,0) -- (-1/1.41*5 + 2, 1/1.41*5) node[above right,midway] {$R_n$};
\draw [-] (0,0) -- (2,0) node[below] {};
\draw (1,0) node[below] {$r$};
\draw [dashed] (0,0) -- (1/1.41*5 , -1/1.41*5) node[below left ,midway] {$R_n$};

\draw (0,0) node[below,scale=0.8]{$0$} node[scale=0.8] {$\bullet$} ;
\draw (2,0) node {$\times$};
\draw (0.4,0) arc (0:115:0.4);
\draw (60:0.8) node[scale = 0.7] {$\theta_r(y)$};

\draw [-] (0,0) -- (-1/1.41*5 + 2, 1/1.41*5) node[below left,midway] {$y$};
\draw (-1/1.41*5 + 1.5, 1/1.41*5) node[above,scale=0.75] {$(y,\theta_r(y))$};
\draw (-1/1.41*5 + 2, 1/1.41*5) node[scale=0.75] {$\bullet$};
\draw [-] (0,0) -- (-1/1.41*5 + 2, -1/1.41*5) node[below right,midway] {$y$};
\draw (-1/1.41*5 + 1.5, -1/1.41*5) node[below,scale=0.75] {$(y,-\theta_r(y))$};
\draw (-1/1.41*5 + 2, -1/1.41*5) node[scale=0.75] {$\bullet$};
\end{tikzpicture}
\end{center}
\caption{Representation of~$\theta_r(y)$ (in a Euclidean setting) } \label{fig:theta_y_r}
\end{figure}

To estimate the degrees of the closest nodes to the centre~$(\Xo{1},\Xo{2},\dots,\Xo{k})$ (with~$k$ fixed), we need an estimate of~$\munB{r}{R_n}$ that holds for~$r$ in intervals containing the radii of these nodes with high probability. The convergence of the radii given in Proposition~\ref{prop:convergence_of_curveR} shows that the choice of this interval highly depends on the value of $\alpha$: going from small radii when~$\al < 1/2$ to large radii for~$\al > 1/2$. This requires treating the three regimes separately.

In the case~$\al > 1/2$, we recall the following crucial estimate for the angle~$\theta_r(y)$ from~\cite[Lemma 3.1]{GugPanPet2012}.
\begin{lem}\label{lem:estimate_theta_al_large}
Suppose~$\al > 1/2$. For~$r \in [0,R_n)$ and~$y \in [R_n - r,R_n)$,
\begin{align*}
\theta_r(y) = 2e^{(R_n-r-y)/2} (1+O(e^{R_n-r-y})).
\end{align*}
\end{lem}
It is shown in~\cite{GugPanPet2012} how this estimate can be combined with~\eqref{formula:munB} to approximate quantities related to the measure of the ball~$\Br{r}{R_n}$. The first estimate of the following lemma is a weaker but more convenient version of the approximation of~$\munB{r}{R_n}$ given in~\cite[Lemma 3.2]{GugPanPet2012}. The second estimate of the lemma shows that most of the mass of~$\Br{r}{R_n}$ is contained in the part of the ball that is outside of the disc of radius~$(1-\ep)R_n$. Which is in line with the fact that most of the nodes of~$\RHG$ concentrate near the boundary of~$\Br{0}{R_n}$.

\begin{lem}\label{lem:approx_munB_al_large}
Suppose~$\al > 1/2$ and set~$\ca \coloneqq \frac{2\al}{\pi(\al-1/2)}$. Fix~$\eta > 0$ and a sequence~$u_n$ that diverges to~$+\infty$. For~$r \in [u_n,R_n)$ and~$x \in [0,(1-\eta)R_n]$, we have,
\begin{align}\label{approx:munBr}
\mu_n(\Br{r}{R_n}\setminus \Br{0}{x}) &= \ca e^{-r/2}(1+o(1)) 
\end{align}
Fix~$\ep \in (0,1)$ and set~$R_n^{\ep} \coloneqq (1-\ep)R_n$, we have for~$r \in [\ep R_n,R_n)$,
\begin{align}\label{approx:munBr_cut_Rnp}
\mu_n(\Br{r}{R_n} \cap \Br{0}{R_n^{\ep}}) = O(e^{-r/2 - \ep (\al - 1/2)R_n}).
\end{align}
\end{lem}
\begin{proof}
Note that since the measure~$\mu_n$ is supported on the ball~$\munB{0}{R_n}$, it holds that $\mu_n(\Br{r}{R_n} \setminus \Br{0}{x}) = \mu_n((\Br{r}{R_n} \cap \Br{0}{R_n}) \setminus \Br{0}{x})$. The proof of~\eqref{approx:munBr} follows directly by showing that for~$r \in [u_n,R_n)$ and~$x \in [0,(1-\eta)R_n]$, the~$O$-terms and the term~$e^{-(\al-1/2)(R_n-x)}$ appearing in the approximations of~$\mu_n((\Br{r}{R_n} \cap \Br{0}{R_n}) \setminus \Br{0}{x})$ given in~\cite[Lemma 3.2]{GugPanPet2012} are of order~$o(1)$. For the proof of~\eqref{approx:munBr_cut_Rnp}, we begin with the following formula, which is obtained in a similar fashion to~\eqref{formula:munB} and is valid for all~$r \in [0,R_n)$,
\begin{align*}
\mu_n(\Br{r}{R_n} \cap \Br{0}{R_n^{\ep}}) 
&= \frac{1}{\pi}\int_0^{R_n^{\ep}}\theta_{r}(y) \rho_n(y) dy.
\end{align*}
Using~\eqref{formula:theta_small_y}, it follows that, for~$r \in [\ep R_n,R_n)$,
\begin{align*}
\mu_n(\Br{r}{R_n} \cap \Br{0}{R_n^{\ep}}) 
&= \munB{0}{R_n-r} + \frac{1}{\pi}\int_{R_n-r}^{R_n^{\ep}}\theta_{r}(y) \rho_n(y) dy.
\end{align*}
For the first term, we have, for~$r \in [\ep R_n,R_n)$,
\begin{align*}
\munB{0}{R_n-r} = \frac{\cosh(\al (R_n-r)) - 1}{\cosh(\al R_n) - 1}= O(e^{-\al r}) = O(e^{- r/2 - \ep (\al - 1/2)R_n})
\end{align*}
and for the second term, Lemma~\ref{lem:estimate_theta_al_large} yields, for~$r \in [\ep R_n,R_n)$,
\begin{align*}
\frac{1}{\pi}\int_{R_n-r}^{R_n^{\ep}}\theta_{r}(y) \rho_n(y) dy = O\Big(\int_{R_n-r}^{R_n^{\ep}} e^{(R_n-r-y)/2} e^{\al(y-R_n)} dy\Big) = O(e^{- r/2 - \ep (\al - 1/2)R_n}).
\end{align*}
We conclude that~\eqref{approx:munBr_cut_Rnp} holds in the regime~$r \in [\ep R_n,R_n)$.
\end{proof}

In the case~$\al < 1/2$, we give the following approximation of~$\munB{r}{R_n}$. Our estimate holds for~$r$ in intervals~$J_n$ whose bounds tend to~$0$ more slowly than~$n^{-1}$. By the convergence result of Proposition~\ref{prop:convergence_of_curveR}, the radii of the closest nodes to the centre are of order~$n^{\al-1/2}$, so the approximation is valid in an appropriate regime for estimating the expected degree of the closest node to the centre. Since we look at small values of~$r$, it is not surprising that~$\munB{r}{R_n}$ is close to~$1$. Here, we provide additional information on the rate at which this quantity decreases with~$r$. The result is stated for all possible values of~$\al$, as this comes without extra cost.
\begin{lem}\label{lem:approx_munB_al_small}
Fix~$\al > 0$ and suppose that the intervals~$J_n$ are such that~$J_n \coloneqq [a_n,b_n]$, with~$na_n \to \infty$ and~$b_n \to 0$ . Then, for~$r \in J_n$, we have
\begin{align*}
\munB{r}{R_n} = 1 - \frac{\al r}{\pi} + o(r).
\end{align*}
\end{lem}

\begin{proof}
Let~$r \in J_n$. By (\ref{formula:munB}) and (\ref{formula:theta_small_y}) we get
\begin{align} \label{formula:by_exterior}
\munB{r}{R_n} = 1 - \frac{1}{\pi}\int_{R_n-r}^{R_n} (\pi - \theta_r(y)) \rho_n(y) dy.
\end{align}
Now, to find a good approximation of~$\theta_r(y)$, we use (\ref{formula:theta_big_y}). Note that~$b_n \to 0$, so we can approximate all the hyperbolic terms in~$r$ that appear in the identity given by (\ref{formula:theta_big_y}). This yields, for~$r \in J_n$ and~$y \in [R_n - r,R_n)$,
\begin{align}
\cos(\theta_r(y)) &= \frac{(1+O(r^2))(1+e^{-2y}) - (e^{R_n-y} + e^{-R_n-y})}{(r+O(r^3))(1-e^{-2y})}. \label{approxi:cos}
\end{align}
Moreover, it holds that
\begin{align*}
e^{-y} &\leq e^{-R_n+b_n} = O(n^{-2}).
\end{align*}
Since~$n a_n \to +\infty$, it follows that
\begin{align*}
e^{-y} = o(r^2). 
\end{align*}
Reporting this in (\ref{approxi:cos}) yields, for~$r \in J_n$ and~$y \in [R_n - r,R_n)$,
\begin{align}
\cos(\theta_r(y)) &= -\frac{R_n-y}{r} + O(r).\label{formula:approx_cos_theta}
\end{align}
For~$r \in J_n$, we define
\begin{align}
U_n(r) \coloneqq 1 - \frac{1}{\pi}\int_{R_n-r}^{R_n} \Big(\pi - \arccos\Big(-\frac{R_n-y}{r}\Big)\Big) \rho_n(y) dy. \label{formula:def_En}
\end{align}
Let us show that~$U_n(r)$ is a good approximation of~$\munB{r}{R_n}$. Since~$\arccos$ is~$1/2$-Hölder, it follows from~\eqref{formula:approx_cos_theta} that, for~$r \in J_n$ and for~$y \in [R_n - r,R_n)$,
\begin{align*}
\theta_r(y) = \arccos\Big(-\frac{R_n-y}{r}\Big) + O\big(\sqrt{r}\big).
\end{align*}
Combining this estimate of~$\theta_r(y)$ with (\ref{formula:by_exterior}) yields, for~$r \in J_n$,
\begin{align}\label{formula:approx_by_E}
\munB{r}{R_n} &= U_n(r) + O\Big(\sqrt{r} \int_{R_n-r}^{R_n} \rho_n(y) dy \Big) \notag\\
              &= U_n(r) + O\big(r^{3/2} \big).
\end{align}
Now, it remains to estimate~$U_n(r)$. Let~$n \in \Nent^*$ and~$r \in J_n$. Integrating (\ref{formula:def_En}) by parts gives
\begin{align*}
U_n(r) &=1 - \frac{1}{\pi} \int_0^1 \frac{1}{\sqrt{1-z^2}}\frac{\cosh(\al R_n) - \cosh(\al (R_n - rz))}{\cosh(\al R_n)-1} dz.
\end{align*}
Furthermore, the mean value theorem yields, for~$r \in J_n$ and~$z \in [0,1]$,
\begin{align*}
\frac{\cosh(\al R_n) - \cosh(\al (R_n - rz))}{\cosh(\al R_n)-1} = \al r z (1+o(1)),
\end{align*}
so
\begin{align*}
U_n(r) &= 1 - \frac{\al r}{\pi} \int_0^1 \frac{z}{\sqrt{1-z^2}} ( 1 + o_{r \in J_n, z \in [0,1]}(1)) dz \\
       &= 1 - \frac{\al r}{\pi} + o_{r \in J_n}(r).
\end{align*}
Combining this with (\ref{formula:approx_by_E}) yields,
\begin{align*}
\munB{r}{R_n} = 1 - \frac{\al r}{\pi} + o_{r \in J_n}(r).
\end{align*}
\end{proof}

In the case~$\al = 1/2$, the radii of the closest nodes to the centre (without renormalisation) converge in distribution (see Proposition~\ref{prop:convergence_of_curveR}), thus we need to approximate~$\munB{r}{R_n}$ for fixed~$r$. This is the purpose of the following lemma.

\begin{lem}\label{lem:conv_munB_n}
Suppose~$\al > 0$ is fixed. Then,
\begin{align*}
\munB{r}{R_n} \tni V_\al(r) \quad \mbox{uniformly for } r \in [0,\infty),
\end{align*} 
where~$V_\al$ is a decreasing diffeomorphism from~$[0,\infty)$ to~$(0,1]$ defined by:
\begin{align}\label{def:V_al}
\forall r > 0,\;V_\al(r) \coloneqq \frac{1}{\pi}\int_0^1 \arccos\left(\max\left(-1,\frac{\cosh(r) - x^{-1/\al}}{\sinh(r)}\right)\right) dx.
\end{align}
\end{lem}

\begin{proof}
Let us fix~$\al > 0$ and~$r > 0$. For all~$n \in \Nent^*$ and~$x \in (0,1)$, we define~$h_n(x)$ as the only positive real number such that
\begin{align*}
\frac{\cosh(\al h_n(x)) - 1}{\cosh(\al R_n) - 1} = x.
\end{align*}
By the change of variable~$x = \frac{\cosh(\al y) - 1}{\cosh(\al R_n) - 1}$  in~(\ref{formula:munB}), we obtain
\begin{align}\label{formula:munB_renormalised}
\munB{r}{R_n} = \frac{1}{\pi} \int_0^1 \theta_{r}(h_n(x)) dx.
\end{align}
Let us fix~$x \in (0,1)$. By the definition of~$h_n$, we have
\begin{align*}
h_n(x) = R_n + \frac{\log(x)}{\al} + o(1).
\end{align*}
Combining this with the expression of~$\theta_r(y)$ given by (\ref{formula:theta_gen}) yields
\begin{align*}
\theta_{r}(h_n(x)) \tni \arccos\left(\max\left(-1,\frac{\cosh(r) - x^{-1/\al}}{\sinh(r)}\right)\right).
\end{align*}
So, by dominated convergence, the functions~$r \mapsto \munB{r}{R_n}$ converge pointwise towards the function~$V_\al$ defined by~\eqref{def:V_al}. The fact that~$V_\al$ is a decreasing diffeomorphism from~$[0,+\infty)$ to~$(0,1]$ follows directly from its expression. Combining this with the fact that the functions~$r \mapsto \munB{r}{R_n}$ are decreasing and take values in~$(0,1]$, we conclude that the convergence is necessarily uniform on~$[0,\infty)$.
\end{proof}

We conclude this section by providing estimates for the volume of balls centred at the origin. These volumes are much easier to obtain than those of the balls studied earlier, as integrating the density~$\rho_n$ directly gives, for all~$r \in (0,R_n)$,
\begin{align*}
\munB{0}{r} = \frac{\cosh(\al r) - 1}{\cosh(\al R_n) - 1}.
\end{align*}
Approximating the hyperbolic terms, we get the following estimates.
\begin{lem}\label{lem:estimates_munB_0}
Let us fix~$\al > 0$ and a sequence~$u_n \to 0$. We have for~$r \in [0,u_n]$,
\begin{align}
\munB{0}{r} &= (\al r)^2e^{-\al R_n}(1+o(1)). \label{approx0:0,r}
\end{align}
Let us fix~$\al > 0$ and a sequence~$v_n \to \infty$. We have for~$r \in [v_n,R_n)$,
\begin{align}
\munB{0}{r} &= e^{\al(r-R_n)}(1+o(1)). \label{approx0:infty,r}
\end{align}
\end{lem}

\section{Constant rank Ordering (Proof of Theorem~\ref{thm:constant_rank_ordering})} \label{section:constant_rank_ordering}

The results of the two previous sections give enough information on the degrees of the closest nodes to the centre to proceed with the proof of Theorem~\ref{thm:constant_rank_ordering}. Indeed, the convergence of the (normalised) point process of the radii, as stated in Proposition \ref{prop:convergence_of_curveR}, also gives estimates for the radius gap between consecutive nodes, $r(\Xo{i+1}) - r(\Xo{i})$. This, combined with the volume estimates of the previous part allows us to get a lower bound on the difference of the expected degrees of the nodes~$\Xo{i}$ and~$\Xo{j}$ with~$i < j$. The following lemma provides a way to translate this information into a bound on the probability that~$\Xo{i}$ has smaller degree than~$\Xo{j}$. Note that the two variables~$A$ and~$B$ are not required to be independent. Thus~$A$ and~$B$ may be legitimately replaced by two node degrees that we want to compare. In the following, we denote by~$\Binlaw{m}{p}$ the binomial distribution with parameters~$m \in \Nent^*$ and~$p \in [0,1]$.

\begin{lem}\label{lem:wrong_order_chernoff}
Let~$A \sim \Binlaw{n}{a}$ and~$B \sim \Binlaw{n}{b}$. If~$0 < b \leq a$, then
\begin{align*}
\pr{A \leq B} \leq 2\exp\Big(-\tfrac{(a-b)^2n}{8a}\Big).
\end{align*}
\end{lem}

\begin{proof}
By splitting at~$n\frac{a+b}{2}$, we get
\begin{align*}
\pr{A \leq B}
&\leq \pr{B \geq n\tfrac{a+b}{2}} + \pr{A \leq n\tfrac{a+b}{2}}\\
&= \pr{B \geq \left(1+\tfrac{a-b}{2b}\right)nb} + \pr{A \leq \left(1-\tfrac{a-b}{2a}\right)na}.
\end{align*}
For~$Z \sim \Binlaw{n}{p}$ and~$\mu = np$, we recall the multiplicative Chernoff bounds:
\begin{align*}
\forall \dt \geq 0&, \;\pr{Z \geq (1+\dt)\mu} \leq e^{-\frac{\dt^2}{2+\dt}\mu}\\
\forall \dt \in [0,1)&,\; \pr{Z \leq (1-\dt)\mu} \leq e^{-\mu \dt^2/2}.
\end{align*}
Since~$\frac{a-b}{2b} \geq 0$ and~$0 \leq \frac{a-b}{2a} < 1$, we can use these bounds to obtain
\begin{align*}
\pr{A \leq B}
&\leq \exp\Big(-\tfrac{(a-b)^2n}{2(a+3b)}\Big) + \exp\Big(-\frac{(a-b)^2n}{8a}\Big)\\
&\leq  2\exp\Big(-\tfrac{(a-b)^2n}{8a}\Big),
\end{align*}
which proves our claim.
\end{proof}

We are now ready to prove Theorem~\ref{thm:constant_rank_ordering}.

\begin{proof}[Proof of Theorem \ref{thm:constant_rank_ordering}]
Let us fix~$k \in \Nent^*$ and define the ordering event~$O_n$ as: 
\begin{align*}
\forall i > k, \; \deg(\Xo{1}) > \deg(\Xo{2}) > \dots > \deg(\Xo{k}) > \deg(\Xo{i}).
\end{align*}
We want to show that the complementary~$O_n^c$ of this event has probability converging to~$0$. We start with the case~$\al < 1/2$. The case~$\al > 1/2$ can be handled in a similar fashion but we do not present it since the result in this case follows from the stronger result of Theorem~\ref{thm:kn_ordering}. We deal with the case~$\al = 1/2$ at the end of the proof.

Define the localisation event~$L_n$ and the gap event~$G_n$ by
\begin{align*}
L_n &\coloneqq \{t_n \leq r(\Xo{1}) \; \mbox{and} \; r(\Xo{k}) \leq r_n\},\\
G_n &\coloneqq \{\forall i \leq k,\, r(\Xo{i+1}) - r(\Xo{i}) \geq \lb_n\},
\end{align*}
where the sequences~$t_n$,~$r_n$ and~$\lb_n$ are defined for~$n \in \Nent^*$ by
\begin{align*}
t_n \coloneqq n^{\al - 1/2} \log(n)^{-1}, \quad r_n \coloneqq n^{\al - 1/2} \log(n), \quad \lb_n \coloneqq n^{\al - 1/2}\log(n)^{-1} .
\end{align*}
By the convergence result stated in Proposition~\ref{prop:convergence_of_curveR} and the characterisation of convergence in distribution in~$M_p([a,b))$ given in Remark~\ref{rem:thm_conv}, we know that
\begin{align*}
\pr{L_n \cap G_n} \tni 1.
\end{align*}
So, it suffices to show that
\begin{align}\label{conv:suffice_to_show_ordering_al_small}
\pr{O_n^c \cap L_n \cap G_n} \tni 0.
\end{align}
If the event~$O_n^c \cap L_n \cap G_n$ occurs, then there must exist~$i \leq k$ and~$i < j$ such that~$\deg(\Xo{i}) \leq \deg(\Xo{j})$. Since~$L_n$ occurs, we also know that~$t_n \leq r(\Xo{i}) \leq r_n$ and since~$G_n$ occurs, we know that~$r(\Xo{j}) - r(\Xo{i}) \geq \lb_n$. So, the probability to control can be bounded like this (note that the bound does not refer to the reordering of the points)
\begin{align*}
\pr{O_n^c \cap L_n \cap G_n} \leq \pr{\exists i \neq j,\; t_n \leq r(X_i) \leq r_n,\, r(X_j) - r(X_i) \geq \lb_n,\, \deg(X_i) \leq \deg(X_j)}.
\end{align*}
Using a union bound and noting~$\prV{\scriptscriptstyle (X_1,X_2)}{\cdot}$ for the conditional probability with respect to~$(X_1,X_2)$, we finally get
\begin{align*}
\pr{O_n^c \cap L_n \cap G_n} 
\leq n^2\E{ \prV{\scriptscriptstyle(X_1,X_2)}{\deg(X_1) \leq \deg(X_2)} \Indi{t_n \leq r(X_1) \leq r_n, \; r(X_2) - r(X_1) \geq \lb_n} }.
\end{align*}
Conditionally on the couple~$(X_1,X_2)$, the variables~$\deg(X_1)$ and~$\deg(X_2)$ have binomial distributions with~$n-2$ trials and probabilities~$\munB{r(X_1)}{R_n}$ and~$\munB{r(X_2)}{R_n}$, respectively (we neglect the potential connection between~$X_1$ and~$X_2$, since it does not contribute to the inequality between their degrees). Given that~$\munB{r(X_1)}{R_n}$ is larger than~$\munB{r(X_2)}{R_n}$ (see Lemma~\ref{lem:decreasing_munB}), we can apply Lemma~\ref{lem:wrong_order_chernoff} to bound the conditional probability of the previous display. It follows that
\begin{align}
\pr{O_n^c \cap L_n \cap G_n} 
\leq \max\limits_{(s_1,s_2)} 2n^2 \exp\left(- \frac{(n-2)(\munB{s_1}{R_n} - \munB{s_2}{R_n})^2}{8\munB{s_1}{R_n}}\right),\label{ineq:bound_on_Oc_al_small}
\end{align}
where the maximum is taken over the couples~$(s_1,s_2)$ belonging to the set
\begin{align*}
E_n \coloneqq \{(s_1,s_2) \in [0,R_n)^2 \,|\, t_n \leq s_1 \leq r_n,\, s_2 - s_1 \geq \lb_n\}.
\end{align*}

Using Lemma~\ref{lem:approx_munB_al_small} to approximate the volume of the balls~$\munB{s_j}{R_n}$ for~$j=1,2$, we get~$K > 0$ such that, for large~$n$ and~$(s_1,s_2) \in E_n$,
\begin{align*}
\frac{(n-2)(\munB{s_1}{R_n} - \munB{s_2}{R_n})^2}{8\munB{s_1}{R_n}} \geq K n\lb_n^2 = K n^{2\al} \log(n)^{-2}.
\end{align*}
Reporting this in~\eqref{ineq:bound_on_Oc_al_small} yields
\begin{align}
\pr{O_n^c \cap L_n \cap G_n} 
&\leq 2n^2 e^{-K n^{2\al} \log(n)^{-2}} \tni 0,
\end{align}
which proves~\eqref{conv:suffice_to_show_ordering_al_small} and concludes the proof of this case.

Now, let us proceed with the case~$\al = 1/2$. We take~$\ep > 0$ and for all~$\delta,\lb > 0$, we introduce the following set:
\begin{align*}
\mathcal{L}(\lb,\delta) \coloneqq \{r \geq 0 : V_\al(r)- V_\al(r+\lb) > \delta\},
\end{align*}
where~$V_\al$ is the function defined in Lemma~\ref{lem:conv_munB_n}.
In this case, we define the localisation event~$L_n$ and the gap event~$G_n$ by
\begin{align*}
L_n &\coloneqq \{\forall i \leq k,\, r(\Xo{i}) \in \mathcal{L}(\lb,\dt)\}\\
G_n &\coloneqq \{\forall i \leq k,\, r(\Xo{i+1}) - r(\Xo{i}) \geq \lb\},
\end{align*}
where the parameters~$\lb$ and~$\dt$ will be fixed later. Since the function~$V_{\al}$ is strictly decreasing, we find that, for every~$\lb > 0$, the set~$\bigcup_{\delta > 0} \mathcal{L}(\lb,\delta)$ covers the entire interval~$[0,+\infty)$.
Moreover, the convergence result stated in Proposition~\ref{prop:convergence_of_curveR} (without normalisation in the case~$\al = 1/2$) and the characterisation of convergence in distribution in~$M_p([a,b))$ given in Remark~\ref{rem:thm_conv}, imply that the vector~$(r(\Xo{i}),\, i \leq k+1)$ converges in distribution. So, we can fix~$\lb > 0$ and~$\delta > 0$ such that 
\begin{align*}
\pr{L_n \cap G_n} \geq 1-\ep.
\end{align*}
Therefore, it remains to prove that
\begin{align}\label{conv:suffice_to_show_ordering_al_crit}
\pr{O_n^c \cap L_n \cap G_n} \tni 0.
\end{align}
Proceeding exactly as we did for proving~\eqref{ineq:bound_on_Oc_al_small} in the previous case, we get 
\begin{align}
\pr{O_n^c \cap L_n \cap G_n} 
&\leq \max\limits_{(s_1,s_2)} 2n^2 \exp\left(- \frac{(n-2)(\munB{s_1}{R_n} - \munB{s_2}{R_n})^2}{8\munB{s_1}{R_n}}\right),\label{ineq:bound_on_Oc_al_crit}
\end{align}
where the maximum is taken over the couples~$(s_1,s_2)$ belonging to the set 
\begin{align*}
E_n \coloneqq \{(s_1,s_2) \in [0,R_n)^2 \,|\, s_1 \in \mathcal{L}(\lb,\dt),\, s_2 - s_1 \geq \lb\}.
\end{align*}
By the definition of~$\mathcal{L}(\lb,\dt)$ and the uniform convergence of~$\munB{r}{R_n}$ towards~$V_{\al}(r)$ (see Lemma \ref{lem:conv_munB_n}), it follows that, for~$n$ sufficiently large,
\begin{align*}
\forall (s_1,s_2) \in E_n,\; \munB{s_1}{R_n} - \munB{s_2}{R_n} \geq \delta/2.
\end{align*}
Reporting this in~\eqref{ineq:bound_on_Oc_al_crit} and bounding~$\munB{s_1}{R_n}$ by~$1$ proves~\eqref{conv:suffice_to_show_ordering_al_crit} and concludes the proof of this case.

\end{proof}

\section{Convergence of the Degrees (Proof of Theorem~\ref{thm:conv_distr_deg})} \label{section:conv_distr_deg}

The key to proving the convergence of the largest degrees is to show that the degrees of the closest nodes to the centre concentrate around their conditional expected values, given the positions of the corresponding nodes. This implies that the point process of the degrees is comparable to the image of the point process of the radii under the mapping~$r \mapsto \munB{r}{R_n}$. The conclusion then follows from the convergence of the point process of the radii and the continuous mapping theorem.
\begin{proof}[Proof of Theorem~\ref{thm:conv_distr_deg}]
Let us begin with the case~$\al > 1/2$. The convergence of the normalised degree process towards a point process~$\calP$ will be stated if we prove that for all~$k$, the normalised vector of the~$k$ highest degrees converges in distribution to the vector of the~$k$ largest points of~$\calP$ (see Remark~\ref{rem:thm_conv}). By Theorem~\ref{thm:constant_rank_ordering}, with high probability, the vector of the~$k$ highest degrees is the vector of the degrees of the~$k$ closest points to the centre (in the same order). So, to prove the convergence of the normalised degree process towards a point process~$\calP$, it suffices to prove that, for all~$k$,
\begin{align}\label{conv:condition_conv_deg_al_large}
n^{-1/(2\al)}(\deg(\Xo{1}),\deg(\Xo{2}), \dots \deg(\Xo{k})) \cLaw (Y_1,Y_2,\dots,Y_k),
\end{align}
where the variables~$Y_1 \geq \dots \geq Y_k$ are the~$k$ largest points of the point process~$\calP$. For all~$s \in \mathbb{R}$, we set
\begin{align*}
p_n(s) \coloneqq \frac{\mu_n(\Br{s}{R_n} \setminus \Br{0}{s})}{1 - \munB{0}{s}} \Indi{s \in [0,R_n)}
\end{align*}
and we define the vectors~$\Dt_k^n$,~$\widetilde{\Dt_k^n}$ and~$W_k^n$ by
\begin{align*}
\Dt_k^n &\coloneqq (\deg(\Xo{1}),\deg(\Xo{2}) \dots ,\deg(\Xo{k})) \\
\widetilde{\Dt_k^n} &\coloneqq (\degtld(\Xo{1}),\degtld(\Xo{2}) \dots ,\degtld(\Xo{k}))\\ 
W_k^n &\coloneqq ((n-i) p_n(r(\Xo{i})),\, 1 \leq i \leq k),
\end{align*}
where~$\degtld(\Xo{k})$ denotes the number of neighbours of the point~$\Xo{k}$ that are in the annulus~$\Cr(r(\Xo{k}),R_n)$. We first show that~$n^{-1/(2\al)}\Dt_k^n$ can be approximated by~$n^{-1/(2\al)}W_k^n$ and in a second step we will prove the convergence of~$n^{-1/(2\al)}W_k^n$. 

Since the difference~$\Dt_k^n - \widetilde{\Dt_k^n}$ is bounded by~$k$ it is clear that 
\begin{align}\label{conv:equivalence_two_degrees_al_large}
n^{-1/(2\al)}(\Dt_k^n - \widetilde{\Dt_k^n}) \cPr 0.
\end{align}
For all~$i$, conditionally on~$\Xo{i}$, the~$n-i$ nodes with radius larger than~$r(\Xo{i})$ are independently and identically distributed according to the restriction of~$\mu_n$ to the annulus~$\Cr(r(\Xo{i}),R_n)$. So, conditionally on~$\Xo{i}$, the variable~$\degtld(\Xo{i})$ is distributed according to the binomial distribution:
\begin{align}\label{expr:conditional_distr_Xr}
\Binlaw{n-i}{p_n(r(\Xo{i}))}.
\end{align}
Let us fix~$K > 0$ and denote by~$L_n$ the event that all the variables~$r(\Xo{i}),\, i \leq k$ belong to the interval~$[(1-\tfrac{1}{2\al})R_n - K,(1-\tfrac{1}{2\al})R_n + K]$. Let us fix~$\ep > 0$. By the convergence of the radii stated in Proposition~\ref{prop:convergence_of_curveR}, we can choose~$K > 0$ such that for sufficiently large~$n$, the event~$L_n$ occurs with probability greater than~$1 - \ep$. By the volume estimates~\eqref{approx:munBr} and~\eqref{approx0:infty,r}, we know that there exists~$K' > 1$ such that, on the event~$L_n$, for large~$n$ and~$i \leq k$, we have
\begin{align}\label{ineq:bound_on_E_deg}
\tfrac{1}{K'}n^{1/(2\al)} \leq \E{\degtld(\Xo{i}) \,\big|\, \Xo{i}} \leq K'n^{1/(2\al)}.
\end{align}
Let us fix~$\dt > 0$. Denoting by~$\mathbb{P}_{L_n}$ the probability measure conditioned on the event~$L_n$, it follows from~\eqref{expr:conditional_distr_Xr} and \eqref{ineq:bound_on_E_deg} that, for large~$n$ and~$i \leq k$,
\begin{multline*}
\prV{L_n}{\big|\degtld(\Xo{i}) - (n-i)p_n(r(\Xo{i}))\big| \geq n^{1/(2\al)} \delta }\\
\leq \prV{L_n}{\Big|\degtld(\Xo{i}) - \E{\degtld(\Xo{i})\,\big|\,\Xo{i}}\Big| \geq  \tfrac{\dt \E{\degtld(\Xo{i})\,|\,\Xo{i}}}{K'}}.
\end{multline*}
Combining this with a Chernoff bound and using~\eqref{ineq:bound_on_E_deg} again yields
\begin{align*}
&\prV{L_n}{\big|\degtld(\Xo{i}) - (n-i)p_n(r(\Xo{i}))\big| \geq n^{1/(2\al)} \delta} \tni 0.
\end{align*}
From this and~\eqref{conv:equivalence_two_degrees_al_large} we conclude that 
\begin{align}\label{conv:variation_alpha_large}
n^{-1/(2\al)}(\Dt_k^n - W_k^n) \cPr 0.
\end{align}
Now, it remains to prove the convergence in distribution of the vector~$n^{-1/(2\al)} W_k^n$. For this purpose, we remark that
\begin{align}\label{expr:weight_from_radii_alpha_large}
n^{-1/(2\al)} W_k^n = \varphi_n\big(r(\Xo{i}) - (1-\tfrac{1}{2\al})R_n,\, 1 \leq i \leq k\big),
\end{align}
where for all~$n$, the application~$\varphi_n$ is defined from~$\mathbb{R}^k$ to~$\mathbb{R}^k$ by
\begin{align*}
\varphi_n\big(z_i, 1 \leq i \leq k\big) \coloneqq \big(n^{-1/(2\al)}(n-i)p_n(z_i+(1-\tfrac{1}{2\al})R_n),\, 1 \leq i \leq k\big).
\end{align*}
By the volume estimates~\eqref{approx:munBr} and~\eqref{approx0:infty,r}, the applications~$\varphi_n$ converges uniformly on every compact of~$\mathbb{R}^k$ towards the function~$\varphi$ defined by
\begin{align*}
\varphi\big(z_i, 1 \leq i \leq k\big) \coloneqq \big(T(z_i),\, 1 \leq i \leq k\big),
\end{align*}
where the application~$T$ is a diffeomorphism from~$[-\infty,\infty)$ to~$(0,\infty]$ given by
\begin{align*}
T(z) \coloneqq \ca \nu^{1-1/(2\al)} e^{-z/2}.
\end{align*}
Let us denote by~$Z_1 \leq \dots \leq Z_k$ the~$k$ smallest points of the point process~$\eta_{\gamma_3}$ (see Proposition~\ref{prop:convergence_of_curveR} for the definition of~$\gamma_3$). By the convergence of the radii stated in Proposition~\ref{prop:convergence_of_curveR}, the vector~$\big(r(\Xo{i}) - (1-\tfrac{1}{2\al})R_n, 1 \leq i \leq k\big)$ converges in distribution to~$(Z_i, 1 \leq i \leq k)$. Consequently, from~\eqref{expr:weight_from_radii_alpha_large} and the continuous mapping theorem (see for example~\cite[Theorem 5.27]{Kal2021} for a version of the continuous mapping theorem with mappings depending on~$n$), we obtain that 
\begin{align}\label{conv:weight_alpha_large}
n^{-1/(2\al)}W_k^n \cLaw \varphi\big(Z_i,\, 1 \leq i \leq k\big),
\end{align}
Combining~\eqref{conv:variation_alpha_large} and~\eqref{conv:weight_alpha_large} with Slutsky theorem finally gives
\begin{align*}
n^{-1/(2\al)}\Dt_k^n \cLaw \varphi\big(Z_i,\, 1 \leq i \leq k\big).
\end{align*}
This proves that~\eqref{conv:condition_conv_deg_al_large} holds for all~$k$ with the limit measure~$\eta_{\gamma_3} \circ T^{-1}$, so
\begin{align*}
\sum\limits_{i=1}^n \dt_{n^{-1/(2\al)}\deg(X_i^n)} \cLaw \eta_{\gamma_3} \circ T^{-1},\quad \mbox{in } M_p((0,\infty]),
\end{align*}
By classical results on transformations of Poisson processes, the point process~$\eta_{\gamma_3} \circ T^{-1}$ is a Poisson process with an intensity measure whose density~$g_3$ with respect to the Lebesgue measure is given by
\begin{align*}
\forall y \in (0,\infty],\, g_3(y) = \frac{\gamma_3(T^{-1}(y))}{|T'(T^{-1}(y))|} = 2\al (\ca \nu)^{2\al} y^{-2\al-1}.
\end{align*}
This concludes the proof for this case. The proofs for the two other regimes follow the same method. We will omit the parts that are too much similar to the proof of the previous case.

In the case~$\al = 1/2$, we want to prove convergence in distribution of the vector
\begin{align*}
n^{-1}(\deg(\Xo{1}),\deg(\Xo{2}) \dots \deg(\Xo{k})).
\end{align*}
We define the functions~$p_n$ and the vectors~$\Dt_k^n$ and~$W_k^n$ like in the previous case. Using the convergence of the radii stated in Proposition~\ref{prop:convergence_of_curveR}, along with the volume estimate given by~\eqref{approx0:infty,r} and Lemma~\ref{lem:conv_munB_n}, one get that, conditionally on the position of~$\Xo{i}$, with high probability, the expected value of~$\degtld(\Xo{i})$ is comparable to~$n$. So, by copying the proof of~\eqref{conv:variation_alpha_large}, we get in this case
\begin{align*}
n^{-1}(\Dt_k^n - W_k^n) \cPr 0.
\end{align*}
It remains to prove the convergence in distribution of the vector~$n^{-1}W_k^n$. For this purpose, we remark that
\begin{align}\label{expr:weight_from_radii_al_crit}
n^{-1} W_k^n = \varphi_n\big(r(\Xo{i}),\, 1 \leq i \leq k\big),
\end{align}
where the application~$\varphi_n$ is defined from~$\mathbb{R}_+^k$ to~$\mathbb{R}_+^k$ by
\begin{align*}
\varphi_n\big(z_i,\, 1 \leq i \leq k\big) \coloneqq \big(p_n(z_i), 1 \leq i \leq k\big).
\end{align*} 
By the volume estimates given by~\eqref{approx0:infty,r} and Lemma~\ref{lem:conv_munB_n}, the applications~$\varphi_n$ converges uniformly on every compact of~$\mathbb{R}_+^k$ towards the function~$\varphi$ defined by
\begin{align*}
\varphi\big(z_i, 1 \leq i \leq k\big) \coloneqq \big(V_{\scriptscriptstyle 1/2}(z_i),\, 1 \leq i \leq k\big),
\end{align*}
So, combining~\eqref{expr:weight_from_radii_al_crit} with the convergence of the radii to~$\eta_{\gamma_2}$ as we did in the case~$\al > 1/2$, we obtain 
\begin{align*}
\sum\limits_{i=1}^n \dt_{n^{-1}\deg(X_i^n)} \cLaw \eta,\quad \mbox{in } M_p((0,1]),
\end{align*}
where~$\eta$ is a Poisson process with an intensity measure whose density~$g_2$ with respect to the Lebesgue measure is given by
\begin{align*}
\forall y \in (0,1],\, g_2(y) =  \frac{\gamma_2(V_{\scriptscriptstyle 1/2}^{-1}(y))}{|V_{\scriptscriptstyle 1/2}'(V_{\scriptscriptstyle 1/2}^{-1}(y))|} = \nu |(V_{\scriptscriptstyle 1/2}^{-1})'(y)| \sinh(V_{\scriptscriptstyle 1/2}^{-1}(y)/2).
\end{align*}
This concludes the proof of this case.

In the case~$\al < 1/2$, we are interested in proving convergence in distribution of the following vector
\begin{align*}
n^{-(\al + 1/2)}(\deg(\Xo{1})-n,\deg(\Xo{2})-n, \dots \deg(\Xo{k})-n).
\end{align*}
Because of the additional normalisation by~$n$, we shall define the vectors~$\Dt_k^n$ and~$W_k^n$ in a different manner in this case, by letting
\begin{align*}
\Dt_k^n &\coloneqq (\deg(\Xo{1})-n,\deg(\Xo{2})-n \dots \deg(\Xo{k})-n) \\
W_k^n &\coloneqq ((n-i) (p_n\big(r(\Xo{i})\big)-1),\, 1 \leq i \leq k).
\end{align*}
Using the convergence of the radii stated in Proposition~\ref{prop:convergence_of_curveR}, together with the volume estimates given by~\eqref{approx0:0,r} and Lemma~\ref{lem:approx_munB_al_small}, we get that, conditionally on the position of~$\Xo{i}$, with high probability, the expected value of~$\degtld(\Xo{i}) - n$ is comparable to~$-n^{\al + 1/2}$. So, by copying the proof of~\eqref{conv:variation_alpha_large}, we get in this case
\begin{align*}
n^{-(\al + 1/2)}(\Dt_k^n - W_k^n) \cPr 0.
\end{align*}
It remains to prove the convergence in distribution of the vector~$n^{-(\al+1/2)}W_k^n$. For this purpose, we remark that
\begin{align}\label{expr:weight_from_radii_al_small}
n^{-(\al + 1/2)} W_k^n = \varphi_n\big(n^{1/2-\al}r(\Xo{i}),\, 1 \leq i \leq k\big),
\end{align}
where, the application~$\varphi_n$ is defined from~$\mathbb{R}_+^k$ to~$\mathbb{R}_-^k$ by
\begin{align*}
\varphi_n\big(z_i, 1 \leq i \leq k\big) \coloneqq \big(n^{-(\al+1/2)}(n-i) (p_n(n^{\al-1/2} z_i)-1),\, 1 \leq i \leq k\big).
\end{align*} 
By the volume estimates given by~\eqref{approx0:0,r} and Lemma~\ref{lem:approx_munB_al_small}, the applications~$\varphi_n$ converges uniformly on every compact of~$\mathbb{R}_+^k$ towards the function~$\varphi$ defined by
\begin{align*}
\varphi\big(z_i, 1 \leq i \leq k\big) \coloneqq \big(T(z_i),\, 1 \leq i \leq k\big),
\end{align*}
where~$T$ is a diffeomorphism from~$[0,+\infty)$ to~$(-\infty,0]$ given by
\begin{align*}
T(z) \coloneqq - \frac{\al z}{\pi}.
\end{align*}
So, proceeding exactly like in the case~$\al > 1/2$, we obtain 
\begin{align*}
\sum\limits_{i=1}^n \dt_{n^{-(\al + 1/2)}(\deg(X_i^n)-n)} \cLaw \eta,\quad \mbox{in } M_p((-\infty,0]),
\end{align*}
where~$\eta$ is a Poisson process with an intensity measure whose density~$g_1$ with respect to the Lebesgue measure is given by 
\begin{align*}
\forall y \in (-\infty,0],\, g_1(y) =  \frac{\gamma_1(T^{-1}(y))}{|T'(T^{-1}(y))|} = 2 \pi^2 \nu^{2\al} |y|.
\end{align*}
This concludes the proof of this case.

\end{proof}

\section{Ordering/non-Ordering transition (Theorem~\ref{thm:kn_ordering})}\label{section:kn_ordering}

From the proof of Theorem~\ref{thm:constant_rank_ordering}, we observe that the key quantities determining the ordering properties of the node degrees are the differences between the volumes of successive balls,~$\munB{\Xo{i}}{R_n}-\munB{\Xo{i+1}}{R_n}$. We expect the ordering to break around the first~$i$ for which this difference is sufficiently small. The following lemma is a refinement of Lemma~\ref{lem:approx_munB_al_large} that allows to obtain fine estimates on these differences. 

\begin{lem}\label{lem:differential_munBr}
For~$\al > 1/2$, there exists~$r_0 > 0$ such that, for~$r \in (r_0,R_n)$,
\begin{align*}
\tfrac{\partial}{\partial r} \munB{r}{R_n} = - \tfrac{\ca}{2} e^{-r/2} (1 + O(e^{-(\al-1/2)r} + r e^{-r})),
\end{align*}
where~$\ca \coloneqq \frac{2\al}{\pi(\al-1/2)}$ (as in Lemma~\ref{lem:approx_munB_al_large}).
\end{lem}

We explain in Remark~\ref{rem:classical_estimates_not_sufficient} why this estimate is necessary for proving Theorem~\ref{thm:kn_ordering} when~$\al$ is close to~$1/2$. We defer the proof of this lemma to the end of the section and proceed directly with the proof of Theorem~\ref{thm:kn_ordering}.

\begin{proof}[Proof of Theorem \ref{thm:kn_ordering}]
In this proof~$K$ stands for a positive constant (depending only on~$\alpha$ and~$\nu$) whose value may change throughout the proof. We define the sequence~$k_n$ as in the statement of the Theorem. The proof of the two assertions~\eqref{assert:ordering_up_to_k_n} and~\eqref{assert:no_ordering_after_n_beta} are done separately. \medskip

\textbf{Ordering up to rank~$k_n$ (Proof of~\eqref{assert:ordering_up_to_k_n})}\medskip

\noindent~$\bullet$ \textbf{First step: Localisation of the~$k_n$ first nodes}\\
Let~$w_n \coloneqq \log(\log(n))$ and define the localisation event~$L_n$ by
\begin{align*}
L_n \coloneqq \{t_n \leq r(\Xo{1}) \; \mbox{and} \; r(\Xo{k_n}) \leq r_n\},
\end{align*}
where
\begin{align*}
t_n \coloneqq (1-\tfrac{1}{2\alpha})R_n - w_n \quad \mbox{and} \quad r_n \coloneqq t_n +  \tfrac{\beta}{\alpha}\log(n).
\end{align*}

Let us prove that, with this choice of~$t_n$ and~$r_n$, the event~$L_n$ is realised with high probability. We already know from Proposition~\ref{prop:convergence_of_curveR} that~$t_n \leq r(\Xo{1})$ holds with high probability. On the other hand, we note that~$r(\Xo{k_n}) \leq r_n$ holds if and only if the number of nodes in the ball~$\Br{0}{r_n}$ is larger than~$k_n$. The number of points falling in this ball has a binomial distribution with expected value
\begin{align*}
n \munB{0}{r_n} \sni \nu n^{\beta}\log(n)^{-\alpha} \gg k_n.
\end{align*}
Using a Chernoff bound it follows that~$r(\Xo{k_n}) \leq r_n$ holds with high probability, so
\begin{align*}
\pr{L_n} \tni 1.
\end{align*}

\noindent$\bullet$ \textbf{Second step: Proving existence of large gaps between the~$k_n$ first nodes}\\
We define the gap event~$G_n$ by
\begin{align*}
G_n \coloneqq \left\{\forall i \leq k_n,\; r(\Xo{i+1}) - r(\Xo{i}) \geq \lbr{i} \right\},
\end{align*}
where, for all~$n$, the function~$\lb_n$ is defined by
\begin{align*}
\forall s \geq 0,\;\lb_n(s) \coloneqq e^{\alpha(R_n - s)} n^{-(\beta+1)}.
\end{align*}
Let us prove that with this choice of~$\lb_n$, the event~$G_n$ is realised with high probability, which will prove that the radius gaps between the~$n^{\beta}$ first nodes are relatively large. 

The first~$k_n+1$ nodes (i.e.~$\Xo{1},\Xo{2},\dots,\Xo{k_n+1}$) can be sampled in the following way. First, sample~$n$ nodes in $\Br{0}{R_n}$ according to the distribution $\mu_n$, select the closest to the centre as~$\Xo{1}$ and erase the~$n-1$ other points. Next, sample~$n-1$ points according to the restriction of~$\mu_n$ to the annulus~$\Cr(r(\Xo{1}),R_n)$, choose the closest to the centre as~$\Xo{2}$ and erase the other ones. Repeat this process until the $k_n + 1$ first nodes have been sampled (at step~$i+1$ sample~$n-i$ points in the annulus~$\Cr(r(\Xo{i}),R_n)$, choose the closest to the centre as~$\Xo{i+1}$ and erase the other ones). Considering this process, we get
\begin{align*}
\pr{G_n^c,\,L_n} &\leq \sum\limits_{i = 1}^{k_n} \pr{Z_i^n \neq 0,\,L_n},
\end{align*}
where~$Z_i^n$ counts the number of points that fall in the annulus $\Cr\big(r(\Xo{i}),r(\Xo{i}) + \lbr{i}\big)$ at step~$i+1$. Conditionally on the position of~$\Xo{i}$, the variable~$Z_i^n$ follows a binomial distribution with~$n-i$ trials and probability~$p_i^n$ given by 
\begin{align*}
p_i^n &\coloneqq \frac{\mu_n\big(\Cr(r(\Xo{i}),r(\Xo{i})+\lbr{i})\big)}{\mu_n\big(\Cr(r(\Xo{i}),R_n)\big)}.
\end{align*}
From this expression, we conclude that, for~$n$ large enough and~$t_n \leq r(\Xo{i}) \leq r_n$, 
\begin{align*}
p_i^n &\leq K \lbr{i} \exp(\alpha (r(\Xo{i}) - R_n)) = K n^{-(\beta+1)}.
\end{align*}
Writing~$\prV{\scriptscriptstyle \Xo{i}}{\cdot}$ for conditional probability with respect to the variable~$\Xo{i}$, we get,
\begin{align*}
\pr{G_n^c \cap L_n} &\leq \sum\limits_{i = 1}^{k_n} \E{\prV{\scriptscriptstyle \Xo{i}}{Z_i^n \neq 0} \Indi{t_n \leq r(\Xo{i}) \leq r_n}}\\
&\leq k_n (1 - (1-K n^{-(\beta+1)})^{n})\\
&\sni K \log(n)^{-2\al} \tni 0.
\end{align*}
Since~$L_n$ occurs with high probability, we conclude that
\begin{align*}
\pr{G_n} \tni 1.
\end{align*}

\noindent$\bullet$  \textbf{Third step: Comparing the successive expected values of the degrees}\\
Let us denote by~$O_n$ the ordering event described in~\eqref{assert:ordering_up_to_k_n}. By the previous two steps, it remains to prove that 
\begin{align}\label{conv:suffice_to_show_ordering_al_large}
\pr{O_n^c \cap L_n \cap G_n} \tni 0.
\end{align}
Proceeding as we did for proving~\eqref{ineq:bound_on_Oc_al_small} in the proof of Theorem~\ref{thm:constant_rank_ordering}, we get the following upper bound
\begin{align}
\pr{O_n^c \cap L_n \cap G_n} 
&\leq \max\limits_{(s_1,s_2)} 2n^2 \exp\left(- \frac{(n-2)(\munB{s_1}{R_n} - \munB{s_2}{R_n})^2}{8\munB{s_1}{R_n}}\right),\label{ineq:bound_on_Oc_al_large}
\end{align}
where the maximum is taken over the couples~$(s_1,s_2)$ belonging to the set 
\begin{align*}
E_n \coloneqq \{(s_1,s_2) \in [0,R_n)^2 \,|\, t_n \leq s_1 \leq r_n,\, s_2 - s_1 \geq \lb_n(s_1)\}.
\end{align*}
We are now interested in bounding the fraction appearing in the exponential term. Using Lemma~\ref{lem:differential_munBr} to estimate the numerator and~\eqref{approx:munBr} to estimate the denominator, we get~$K > 0$ such that, for~$(s_1,s_2) \in E_n$ and~$n$ large, 
\begin{align}
\frac{(n-2)(\munB{s_1}{R_n} - \munB{s_2}{R_n})^2}{8\munB{s_1}{R_n}}
&\geq K \frac{n(\int_{s_1}^{s_2} e^{-y/2}dy)^2}{e^{-s_1/2}} \notag\\
&\geq K ne^{-s_1/2}(1-e^{-(s_2-s_1)/2})^2 \notag\\
&\geq K ne^{-r_n/2}\lb_n(r_n)^2, \label{ineq:bound_on_arg_exp}
\end{align}
Moreover, straightforward computations give
\begin{align*}
ne^{-r_n/2} = K n^{\tfrac{1-\beta}{2\alpha}} \log(n)^{1/2} \quad \mbox{and} \quad   \lb_n(r_n)= K n^{-2\beta} \log(n)^\al
\end{align*}
By choice of~$\beta$, we have~$\frac{1-\beta}{2\al} = 4\beta$. So, reporting the above in~\eqref{ineq:bound_on_arg_exp} and combining with~\eqref{ineq:bound_on_Oc_al_large} gives
\begin{align*}
\pr{O_n^c \cap L_n \cap G_n} 
&\leq 2n^2\exp(- K \log(n)^{1/2+2\al}) \tni 0,
\end{align*}
which proves~\eqref{conv:suffice_to_show_ordering_al_large} and concludes the proof of \eqref{assert:ordering_up_to_k_n}.

\begin{rem}\label{rem:classical_estimates_not_sufficient}
One could obtain a bound similar to~\eqref{ineq:bound_on_arg_exp} by using the volume estimate from~\cite[Lemma 3.2]{GugPanPet2012} instead of Lemma~\ref{lem:differential_munBr} to bound the numerator. However for~$\al \in (1/2,\frac{7 + \sqrt{33}}{16})$, the errors terms appearing in this estimate are larger than~$\lb_n$ which is not sufficiently precise here. The estimate of the differential of~$\munB{r}{R_n}$ in~$r$ given by Lemma~\ref{lem:differential_munBr} is more suitable for measuring volume differences of balls with close radii.
\end{rem}

\textbf{No ordering beyond rank~$n^\beta$ (Proof of~\eqref{assert:no_ordering_after_n_beta})} \medskip

Fix an arbitrary sequence~$v_n$ diverging to~$+\infty$. Without loss of generality, suppose that~$v_n = O(\log(n))$. This assumption on~$v_n$ will be used implicitly during the proof when proving certain bounds. Let~$w_n' \coloneqq \log(v_n)$ and define
\begin{align*}
t_n' &\coloneqq (1-\tfrac{1}{2\al})R_n + \tfrac{\beta}{\al}\log(n) + w_n' \quad \mbox{and} \quad r_n' \coloneqq t_n' + w_n'.
\end{align*}
The point~$\Xo{i}$ and $\Xo{i+1}$ we are seeking have radii that are approximately located in the interval~$[t_n',r_n']$. So, the counterpart of the non-ordering event~\eqref{assert:no_ordering_after_n_beta} for the Poissonised model~$\RHGP$ can be written like this:\\
\emph{With high probability, there exist two nodes~$v$ and~$v'$ in the graph~$\RHGP$, such that
\begin{align}\label{assert:ordering_for_poisson}
t_n' \leq r(v) \leq r(v') \leq r_n' \quad \mbox{and} \quad \deg(v) + \dt_n \leq \deg(v'),
\end{align}
where the sequence~$\dt_n$ is defined by~$\dt_n \coloneqq \sqrt{ne^{-r_n'/2}} = K n^{2 \beta} e^{-w_n'/2}$.}

In the first three steps of the proof, we will work in the Poissonised model~$\RHGP$ and prove this Poissonised statement. The final step of the proof is a de-Poissonisation procedure that gives the result for~$\RHG$. The extra gap of size~$\dt_n$ that appears in the Poissonised statement is crucial for the de-Poissonisation step. 

Let us fix a small~$\ep \in (0,1)$ and set~$R_n^{\ep} \coloneqq (1-\ep)R_n$. For~$x \in \Br{0}{R_n}$, we denote by~$\Brep{x}{R_n}$ the part of the ball~$\Br{x}{R_n}$ that lies beyond the circle of radius~$R_n^{\ep}$, i.e.,
\begin{align*}
\Brep{x}{R_n} \coloneqq \Br{x}{R_n} \cap \Cr(R_n^{\ep},R_n).
\end{align*}
The~$\ep$-degree of a node~$X$ is the number of nodes (excluding~$X$ itself) that are contained in~$\Brep{X}{R_n}$. It is denoted~$\degep(X)$. Since the nodes of~$\RHGP$ concentrate near the boundary of~$\Br{0}{R_n}$, the~$\ep$-degrees provide good estimates of the actual degrees of the nodes. Therefore, we will first focus on~$\ep$-degree rather than degrees directly, as it is easier to get independence results concerning the~$\ep$-degrees. \medskip

\noindent$\bullet$ \textbf{First step: Finding good pairs of candidates for the couple~$(v,v')$}\\
For all~$n$, set~$\lb_n' \coloneqq n^{-2\beta} e^{- \al w_n'}$ and denote by~$C_n$ the following event (see Figure~\ref{fig:event_Cn}): \\
\emph{There exist two sequences~$(v_i)_{\scriptscriptstyle 1 \leq i \leq c_n}$ and~$(v_i')_{\scriptscriptstyle 1 \leq i \leq c_n}$ each containing~$c_n$ nodes of~$\RHGP$, such that~$v_1,\dots,v_{c_n},v_1',\dots,v_{c_n}'$ are pairwise distinct and, for all~$i$,
\begin{align}\label{cond:small_gap}
t_n' \leq r(v_i) \leq r(v_i') \leq r(v_i)+\lb_n' \leq r_n'.
\end{align}
We also require that, for all~$x$ and~$y$ distinct in~$\{v_1,\dots,v_{c_n},v_1',\dots,v_{c_n}'\}$,
\begin{align}\label{cond:disjoint_ep_balls}
\Brep{x}{R_n} \cap \Brep{y}{R_n}  = \emptyset .
\end{align}}Let us prove that this event is realised with high probability. We will prove in the following two steps of the proof that when this event occurs, there is a high probability of finding two nodes satisfying~\eqref{assert:ordering_for_poisson}. We define~$m_n \coloneqq w_n' n^{2\beta}$ and take~$c_n = o(w_n')$ such that~$c_n \to \infty$. For all~$1 \leq j \leq m_n$, let~$I_j^n \coloneqq [t_n' + j \lb_n', t_n' + (j+1)\lb_n')$. Note that all of these intervals are contained in~$[t_n',r_n']$. We denote by~$A_j^n$ the event that there exist at least two nodes of $\RHGP$ with radial coordinates in the interval~$I_j^n$. The number of nodes having a radial coordinate in the interval~$I_j^n$ follows a Poisson distribution with parameter~$n p_j^n$, with~$p_j^n$ given by
\begin{align*}
p_j^n = \frac{\cosh(\al (t_n' + (j+1) \lb_n')) - \cosh(\al (t_n' + j \lb_n'))}{\cosh(\al R_n) - 1} \geq K \lb_n' e^{\al (t_n' - R_n)} = K n^{-(\beta + 1)}.
\end{align*}
It follows that
\begin{align*}
\pr{A_j^n} \geq Kn^{-2\beta}.
\end{align*}
Moreover, the events~$A_j^n$ are independent, so the number of indices~$1 \leq j \leq m_n$ for which the event~$A_j^n$ occurs dominates a binomial distribution with~$m_n$ trials and an expected value equivalent to~$Kw_n'$. Since,~$c_n = o(w_n')$, it follows that, with high probability, we can find~$c_n$ indices~$j$ for which the event~$A_j^n$ is realised. We conclude that, with high probability, there exist two sequences made of distinct nodes~$(v_i)_{1 \leq i \leq c_n}$ and~$(v_i')_{1 \leq i \leq c_n}$ satisfying condition~\eqref{cond:small_gap} of the event~$G_n$.

\begin{figure}
\centering
\renewcommand{\C}{(0,0) circle (3)}
\newcommand{\CC}{(0,0) circle (2.6)}
\begin{tikzpicture}
\begin{scope}
\clip \C;
\node[inner sep=0pt] at (0,0) {\includegraphics[scale = 0.42] {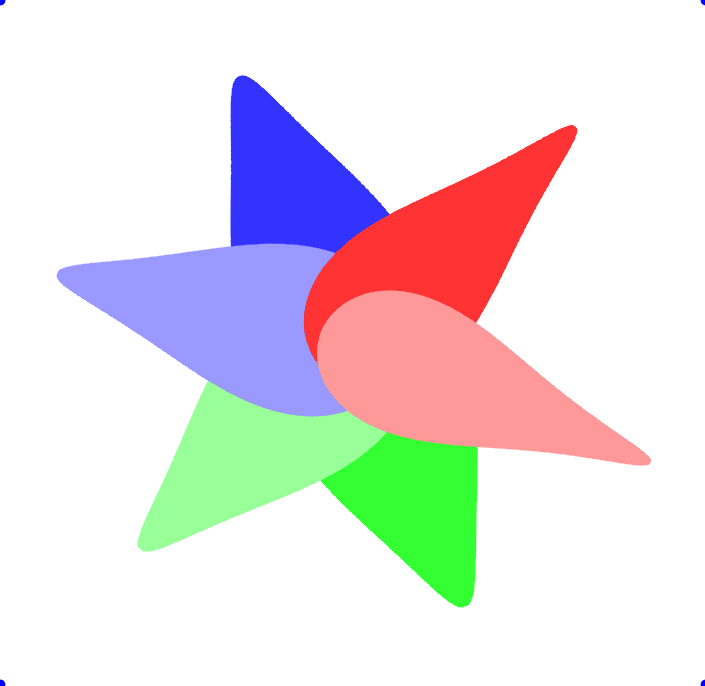}};
\end{scope}
\fill[white,opacity = 0.65] \CC;

\foreach \theta/\r/\label in {-67/1.68/v_1,-136.2/1.8/v_1',113.1/1.77/v_2,166.5/1.83/v_2',44.2/2.02/v_3,-21.7/2.13/v_3'}
{
    \draw (\theta:\r) node[scale = 1] {$\bullet$} node[below left] {$\label$};          
}

\draw \C;
\draw (240:3) node[below left,scale=0.8]  {$B_0(R_n)$};
\draw (0,0) node[scale = 0.8,black] {$\bullet$};
\draw (0,0) node[scale = 0.8, below right,black] {$0$};

\draw[dashed] \CC;
\draw (-90:2.6) node[above,scale=0.8]  {$R_n^{\ep}$};

\end{tikzpicture}
\caption{Depiction of the closeness event~$C_n$. By condition~\eqref{cond:small_gap}, for all~$i$, the radius gap~$r(v_i') - r(v_i)$ is small, ensuring that~$\degep(v_i) \leq \degep(v_i')$ holds with a probability bounded away from~$0$. By condition~\eqref{cond:disjoint_ep_balls}, the portions of the balls~$\Br{v_i}{R_n}$ and~$\Br{v_i'}{R_n}$ that lie beyond the circle of radius~$R_n^{\ep}$ are all disjoints, ensuring the independence of the corresponding~$\ep$-degrees.}\label{fig:event_Cn}
\end{figure}

Let us prove that these nodes also have a high probability of satisfying condition~\eqref{cond:disjoint_ep_balls}. First remark that, by the estimate of~$\theta_r(y)$ (given in Lemma~\ref{lem:estimate_theta_al_large}), if~$\ep$ is small enough, then 
\begin{align*}
\theta_{t_n'}(R_n^{\ep})
= 2\exp\left((R_n-t_n' - R_n^{\ep})/2\right) (1 + O(e^{R_n - t_n' - R_n^{\ep}})) \tni 0.
\end{align*}
So, we may assume that~$c_n$ was chosen such that~$c_n^2 \theta_{t_n'}(R_n^{\ep}) \to 0$. 

Sample all the nodes appearing in~$\Cr(t_n',r_n')$ and suppose that there exist two sequences made of distinct nodes~$(v_i)_{1 \leq i \leq c_n}$ and~$(v_i')_{1 \leq i \leq c_n}$ in~$\Cr(t_n',r_n')$ that satisfy condition~\eqref{cond:small_gap} of the event~$C_n$. Suppose that the angular coordinates of the nodes of these sequences have not been sampled yet. It is valid to suppose that, as condition~\eqref{cond:small_gap} only concerns the radial coordinates of the nodes. We now sample the angular coordinates of these nodes one by one. Each time we sample a new angular coordinate, the probability that it differs by less than~$2\theta_{t_n'}(R_n^{\ep})$ with an already sampled angular coordinate is upper bounded by~$\tfrac{2c_n}{\pi} \theta_{t_n'}(R_n^{\ep})$. So, the probability of getting a pair of angular coordinates that differ by less than~$2\theta_{t_n'}(R_n^{\ep})$ is upper bounded by~$\tfrac{4c_n^2}{\pi}\theta_{t_n'}(R_n^{\ep})$, which tends to~$0$ by our choice of~$c_n$. Thus, with high probability, all the angular coordinates of the nodes~$v_1,\dots,v_{c_n},v_1',\dots,v_{c_n}'$ differ by more than~$2\theta_{t_n'}(R_n^{\ep})$.

For~$1 \leq i \leq c_n$, the angle of the smallest cone (with apex at~$0$) containing~$\Brep{v_i}{R_n}$ is~$2\theta_{r(v_i)}(R_n^{\ep})$. Since~$t_n' \leq r(v_i)$, Lemma~\ref{lem:decreasing_munB} implies that this angle is at most~$2\theta_{t_n'}(R_n^{\ep})$. The same holds for the nodes~$v_i'$. It follows that the sequences~$(v_i)$ and~$(v_i')$ also satisfy condition~\eqref{cond:disjoint_ep_balls}, with high probability. Thus,
\begin{align}\label{conv:high_proba_Cn}
\pr{C_n} \tni 1.
\end{align}

\noindent$\bullet$ \textbf{Second step: Finding a pair of nodes~$(V,V')$ with~$r(V) \leq r(V')$ such that~$\degep(V)$ is small and~$\degep(V')$ is large}\\
Let us sample all the nodes that fall in the annulus~$\Cr(t_n',r_n')$. Since the event~$C_n$ depends only on the point process of the nodes restricted to this region, we can suppose that the points in this annulus are such that the event~$C_n$ occurs. We will denote by~$\mathbb{P}_{\scriptscriptstyle C_n}$ the probability measure conditioned on the event~$C_n$. 

Select two sequences~$(v_i)_{1 \leq i \leq c_n}$ and~$(v_i')_{1 \leq i \leq c_n}$ as provided by the event~$C_n$ (choose them according to a fixed rule to ensure measurability). For all~$1 \leq i \leq c_n$, given the information revealed so far,~$\degep(v_i)$ is a Poisson variable with parameter~$n\mu_n(\Brep{v_i}{R_n})$. For $1 \leq i \leq c_n$, since $r(v_i) \leq r_n'$, it follows from Lemma~\ref{lem:decreasing_munB} that
\begin{align*}
n\mu_n(\Brep{v_i}{R_n}) 
\geq n\mu_n(\Brep{r_n'}{R_n}) = n (\mu_n(\Br{r_n'}{R_n}) - \mu_n(\Br{r_n'}{R_n} \cap \Br{0}{R_n^{\ep}})).
\end{align*}
If~$\ep$ is chosen small enough then~$r_n' \geq \ep R_n$. So, we can use the approximations given by~\eqref{approx:munBr} and~\eqref{approx:munBr_cut_Rnp} to get, uniformly in~$1 \leq i \leq c_n$,
\begin{align*}
n\mu_n(\Brep{v_i}{R_n}) \geq K n e^{-r_n'/2} = K \dt_n^2
\end{align*}
(the sequence~$\dt_n$ is defined below~\eqref{assert:ordering_for_poisson}). The same bound also applies to the points~$v_i'$. By the central limit theorem applied to sums of independent Poisson variables with parameter~$1$, it follows that there exists~$\eta_1 > 0$ such that, for all~$n$ and for all~$1 \leq i \leq c_n$,
\begin{align}\label{ineq:positiv_proba_much_larger}
\prV{\scriptscriptstyle C_n}{\degep(v_i') \geq n\mu_n(\Brep{v_i'}{R_n}) + 2\dt_n} \geq \eta_1.
\end{align}
Similarly, we get~$\eta_2 > 0$ such that, for all~$n$ and~$1 \leq i \leq c_n$,
\begin{align}\label{ineq:positiv_proba_not_large}
\prV{C_n}{\degep(v_i) \leq n\mu_n(\Brep{v_i}{R_n})} \geq \eta_2. 
\end{align}
By condition \eqref{cond:disjoint_ep_balls} of the event~$C_n$ and properties of Poisson point processes, the variables~$\degep(v_1),\dots,\degep(v_{c_n}),\degep(v_i'),\dots,\degep(v_{c_n}')$ are independent. It follows from~$c_n \to \infty$ that, with high probability, we can find~$1 \leq i \leq c_n$ for which both events in~\eqref{ineq:positiv_proba_much_larger} and~\eqref{ineq:positiv_proba_not_large} occur. Define~$i_0$ as the smallest such index and set~$(V,V') = (v_{i_0},v_{i_0}')$. We verify that this defines a random vector~$(V,V')$ on an event of high probability. The definition of $(V,V')$ out of this event is not relevant. \medskip

\noindent$\bullet$ \textbf{Third step: From~$\degep$ to~$\deg$}\\
The previous step would allow us to conclude immediately if we were comparing the~$\ep$-degrees~$\degep$ of the nodes instead of their actual degrees~$\deg$. So, it remains to show that~$\degep$ is very close to~$\deg$. To achieve this, we introduce the following subsets of the ball~$\Br{x}{R_n}$:
\begin{align*}
\Brp{x}{R_n} &\coloneqq \Br{x}{R_n} \setminus (\Cr(t_n',r_n') \cup \Cr(R_n^{\ep},R_n)),\\
\Brpp{x}{R_n} &\coloneqq \Br{x}{R_n} \cap \Cr(t_n',r_n')
\end{align*}
and for a node~$X$, we write~$\degp(X)$ (resp.~$\degpp(X)$) for the number of nodes (excluding~$X$) that are contained in~$\Brp{X}{R_n}$ (resp.~$\Brpp{X}{R_n}$). With these notations, the degree of $X$ can be decomposed as follows:
\begin{align*}
\deg(X) = \degep(X) + \degp(X) + \degpp(X).
\end{align*}
Let us first estimate the error induced by the term~$\degp$. If~$\ep$ is small enough, we can use~\eqref{approx:munBr_cut_Rnp} and get~$\ep' > 0$ such that, for~$n$ large enough,
\begin{align*}
n\mu_n(\Brp{V}{R_n}) \leq  n\mu_n(\Br{V}{R_n} \cap \Br{0}{R_n^{\ep}}) \leq n^{-5\ep'} ne^{-t_n'/2} \leq n^{-4\ep'} \dt_n^2
\end{align*}
(the sequence~$\dt_n$ is defined below~\eqref{assert:ordering_for_poisson}). Since $\Brp{V}{R_n}$ does not intersect the annulus $\Cr(t_n',r_n')$, we have that, conditionally on~$C_n$ and on the position of the node $V$, the variable $\degp(V)$ follows a Poisson distribution with parameter $n\mu_n(\Brp{V}{R_n}) $. Thus, by Chebyshev's inequality,
\begin{align}\label{ineq:degp_V}
\prV{C_n}{\degp(V) \leq n\mu_n(\Brp{V}{R_n}) + n^{-\ep'}\dt_n} \tni 1.
\end{align}
With a similar argument, we also prove that
\begin{align}\label{ineq:degp_Vp}
\prV{C_n}{\degp(V') \geq n\mu_n(\Brp{V'}{R_n}) - n^{-\ep'}\dt_n} \tni 1.
\end{align}
Now, concerning~$\degpp$, a direct computation gives, for $r \in [0,R_n)$,
\begin{align}\label{bound:volume_intermediate_ball}
n\mu_n(\Cr(t_n',r_n')) = o(n^{\beta} \log(n)^{3\al}).
\end{align}
Since $\degpp(V)$ is smaller than the number of nodes that fall in the annulus $\Cr(t_n',r_n')$, it follows from Chebyshev's inequality that
\begin{align}\label{ineq:degpp_V}
\prV{C_n}{\degpp(V) \leq n^\beta \log(n)^{3\al}} \tni 1.
\end{align}
Combining the definition of~$V$ with~\eqref{ineq:degp_V} and~\eqref{ineq:degpp_V}, we get, with high probability,
\begin{align}\label{ineq:low_deg_V}
\deg(V)
&= \degep(V) + \degp(V) + \degpp(V) \notag\\
&\leq n \mu_n(\Brep{V}{R_n}) + n \mu_n(\Brp{V}{R_n}) + n^{-\ep'}\dt_n + n^{\beta}\log(n)^{3 \al} \notag\\
&\leq n \munB{V}{R_n} + o(\dt_n).
\end{align}
Likewise, combining the definition of~$V'$ with~\eqref{ineq:degp_Vp} and~\eqref{bound:volume_intermediate_ball}, we get, with high probability,
\begin{align}\label{ineq:high_deg_Vp}
\deg(V') 
&\geq \degep(V') + \degp(V') \notag\\
&\geq n \mu_n(\Brep{V'}{R_n}) + 2\dt_n + n\mu_n(\Brp{V'}{R_n}) - n^{-\ep'}\dt_n  \notag\\
&\geq n \munB{V'}{R_n}  + 2\dt_n + o(\dt_n).
\end{align}
On the other hand, using that~$r(V') - r(V) \leq \lb_n'$, Lemma~\ref{lem:differential_munBr} yields,
\begin{align*}
n\munB{V}{R_n} - n\munB{V'}{R_n} \leq Kn e^{-t_n'/2} \lb_n' = o(\dt_n).
\end{align*}
Combining this with~\eqref{ineq:low_deg_V} and~\eqref{ineq:high_deg_Vp} proves that, with high probability,
\begin{align}\label{ineq:deg_V_Vp}
\deg(V) + \dt_n \leq \deg(V').
\end{align}
Since~$t_n' \leq r(V) \leq r(V') \leq r_n'$, this concludes the proof of~\eqref{assert:ordering_for_poisson}.\medskip

\noindent$\bullet$ \textbf{Fourth step: de-Poissonisation}\\
Let us now explain how to transfer the result from the Poissonised model~$\RHGP$ to the original model~$\RHG$. We use the standard procedure that consists in coupling the original model with the Poissonised model by sampling the points of~$\RHG$ in the following manner: sample the graph~$\RHGP$ and call~$N_n$ the number of nodes in this graph. If~$N_n > n$, randomly remove~$N_n - n$ nodes from the graph. If~$N_n < n$, add~$n - N_n$ nodes, independently sampled from~$\Br{0}{R_n}$ according to~$\mu_n$ and connect all pairs of nodes that are within distance~$R_n$. The resulting graph follows the same distribution as~$\RHG$ and is referred to as the de-Poissonised graph.

For a node~$x$ that appears in both the Poissonised and de-Poissonised graph, we denote its degree in each graph by~$\deg(x)$ and~$\degtld(x)$, respectively. Fix~$\ep > 0$ small enough. The random variable~$N_n$ follows a Poisson distribution with parameter~$n$. So,
\begin{align*}
\pr{|N_n - n| \geq n^{1/2+\ep}} \tni 0.
\end{align*}
From this, it follows that, with high probability, the nodes~$V$ and~$V'$ are not removed during the de-Poissonisation procedure. For the remainder of the proof, we work on the event that~$V$ and~$V'$ are not removed. In the case where~$N_n < n$, conditionally on the position of~$V$, the variable~$\degtld(V) - \deg(V)$ follows a binomial distribution with~$n - N_n$ trials and probability parameter~$\munB{V}{R_n}$. Thus, with high probability,
\begin{align*}
\degtld(V) - \deg(V) \leq n^{1/2+2\ep} \munB{r_n'}{R_n} \leq K n^{2\beta + 2\ep - 1/2} e^{-w_n'/2} \dt_n = o(\dt_n),
\end{align*}
where the last equality holds if~$\ep$ is chosen sufficiently small (because $\beta < 1/5$). In the case~$N_n > n$, we can also prove with similar arguments that the de-Poissonisation reduces the degree of $V'$ by at most $o(\delta_n)$. Therefore, we finally conclude from~\eqref{ineq:deg_V_Vp} that $\degtld(V) < \degtld(V')$ holds with high probability. In addition, straightforward estimates, using the fact that~$t_n' \leq r(V) \leq r(V') \leq r_n'$, show that in the de-Poissonised graph, the ranks of the nodes~$V$ and~$V'$ in the ranking of the nodes by increasing radii are in the interval~$[n^{\beta},n^{\beta} v_n^{3\al}]$, with high probability. It follows that~\eqref{assert:no_ordering_after_n_beta} holds with high probability (the power~$3\al$ is not problematic, as the sequence~$v_n$ is an arbitrary sequence diverging to~$+\infty$).
\end{proof}

Let us conclude by proving Lemma~\ref{lem:differential_munBr}.

\begin{proof}[Proof of Lemma~\ref{lem:differential_munBr}]
We recall the integral expression~\eqref{formula:munB} of~$\munB{r}{R_n}$:
\begin{align*}
\munB{r}{R_n} = \frac{1}{\pi}\int_0^{R_n}\theta_r(y) \rho_n(y) dy.
\end{align*}
Using the expression of~$\theta_r(y)$ given by~\eqref{formula:theta_gen}, we get
\begin{align*}
\tfrac{\partial}{\partial r} \munB{r}{R_n} = \frac{1}{\pi}\int_{R_n-r}^{R_n} f_{n,r}(y) \rho_n(y) dy,
\end{align*}
where~$f_{n,r}(y) \coloneqq \tfrac{\partial}{\partial r} \arccos(c_{n,r}(y))$, with~$c_{n,r}(y) \coloneqq \frac{\cosh(r)\cosh(y) - \cosh(R_n)}{\sinh(r) \sinh(y)}$. Since the angle~$\theta_r(y)$ is decreasing in~$r$ (see Lemma~\ref{lem:decreasing_munB}), the integral above is always well-defined in~$\mathbb{R} \cup \{-\infty\}$. The contribution of the integral over~$(R_n-r,R_n-r+r_0)$ requires special treatment, because the quantity~$f_{n,r}(y)$ diverges as~$y$ approaches~$R_n-r$. We show at the end of the proof that this contribution is~$O(e^{-\al r})$ (which can be incorporated into the first error term of the formula given in the statement). More precisely, we will prove at the end of the proof that, for~$r \in (r_0,R_n)$,
\begin{align}\label{formula:integral_in_bulk}
\tfrac{\partial}{\partial r} \munB{r}{R_n} = \frac{1}{\pi}\int_{R_n-r+r_0}^{R_n} f_{n,r}(y) \rho_n(y) dy + O(e^{-\al r}).
\end{align}
For now, let us estimate the integrand for~$r \in (r_0,R_n)$ and~$y \in (R_n-r+r_0,R_n)$. A direct computation gives
\begin{align*}
f_{n,r}(y) 
&= -\frac{\cosh(R_n)\cosh(r)-\cosh(y)}{\sinh(r)^2\sinh(y) \sqrt{1-c_{n,r}(y)^2} }.
\end{align*}
Using~$\cosh(x) = e^x(1/2+O(e^{-2x}))$ and~$\sin(x) = e^x(1/2+O(e^{-2x}))$, it follows that
\begin{align}\label{approx:fnr_with_c_gen}
f_{n,r}(y) 
&= -\frac{2e^{R_n-r-y}(1+O(e^{-2r})+O(e^{y-R_n-r}))}{(1+O(e^{-2r} + e^{-2y})) \sqrt{1-c_{n,r}(y)^2} }.
\end{align}
If~$r_0$ is chosen large enough, using~$1/(1+x) = 1+O(x)$ for~$|x| < 1/2$, we get 
\begin{align}\label{approx:fnr_with_c}
f_{n,r}(y) 
&= -\frac{2e^{R_n-r-y}(1+O(e^{y-R_n-r}) + O(e^{-2y}))}{\sqrt{1-c_{n,r}(y)^2} }.
\end{align}
Note that we used that~$y-R_n-r \geq -2r$, to get rid of the~$O(e^{-2r})$ term.
A similar computation gives, for~$r \in (r_0,R_n)$ and~$y \in (R_n-r+r_0,R_n)$,
\begin{align*}
c_{n,r}(y)
&= 1-2e^{R_n-r-y}+O(e^{-2r}+e^{-2y}) .
\end{align*}
So, if~$r_0$ is chosen large enough, we have
\begin{align*}
\frac{1}{\sqrt{1-c_{n,r}(y)^2}} 
&= \frac{e^{-(R_n-r-y)/2}}{2}(1+O(e^{R_n-r-y})).
\end{align*}
Reporting this in (\ref{approx:fnr_with_c}) yields, for~$r \in (r_0,R_n)$ and~$y \in (R_n-r+r_0,R_n)$,
\begin{align*}
f_{n,r}(y) 
&= -e^{(R_n-r-y)/2}(1 + O(e^{R_n-r-y})).
\end{align*}
Estimating the hyperbolic terms in~$\rho_n(y)$, it follows that
\begin{align}\label{approx:fnr}
f_{n,r}(y)\rho_n(y) 
&= - \al e^{(1/2-\al)R_n-r/2}e^{(\al-1/2)y}(1 + O(e^{R_n-r-y}) + O(e^{-\al R_n})).
\end{align}
Solving the integral in~\eqref{formula:integral_in_bulk} using~\eqref{approx:fnr}, without taking into account the error terms, gives
\begin{align*}
-\frac{\al}{\pi} \int_{R_n-r+r_0}^{R_n} e^{(1/2-\al)R_n-r/2}e^{(\al-1/2)y} dy 
&= -\tfrac{\ca}{2} e^{-r/2} (1 + O(e^{-(\al - 1/2)r})).
\end{align*}
If~$\al \neq 3/2$, then the integral over the first error term gives
\begin{align*}
O\Big(\int_{R_n-r+r_0}^{R_n} e^{(1/2-\al)R_n-r/2}e^{(\al-1/2)y} e^{R_n-r-y} dy  \Big) 
&= O(e^{-3r/2}) + O(e^{-\al r}).
\end{align*}
In the case ~$\al = 3/2$, the addition of the error term~$re^{-r}$ in the statement makes it correct. The integral over the second error term (corresponding to the~$O(e^{-\al R_n})$ term of~\eqref{approx:fnr}) is of order~$O(e^{-r/2 - \al R_n})$, which can be incorporated into the first error term of the result. Thus, the computation of the integral appearing in~\eqref{formula:integral_in_bulk} yields the correct estimate. 

It remains to prove~\eqref{formula:integral_in_bulk}. This requires a bound on the speed of divergence of the function~$f_{n,r}(y)$, as~$y$ approaches~$R_n-r$. To obtain this bound, we first need to bound~$c_{n,r}(y)$, for~$y$ in the neighbourhood of~$R_n -r$. In the following,~$K$ stands for a positive constant whose value depends only on~$\al, \nu$ and~$r_0$ and may change throughout the proof. For two functions~$f$ and~$g$ we denote by~$f \wedge g$ (resp.~$f \vee g$) the minimum (resp. maximum) of~$f$ and~$g$. Using the formula~$\cosh(a+b) = \cosh(a)\cosh(b) + \sinh(a)\sinh(b)$ and the mean value theorem to bound the difference~$\cosh(r+y) - \cosh(R_n)$, we obtain, for~$y \in (R_n-r,R_n-r+r_0)$,
\begin{align*}
c_{n,r}(y) 
&= \tfrac{\cosh(r+y)- \sinh(r)\sinh(y) - \cosh(R_n)}{\sinh(r)\sinh(y)}\\
&\geq K(y-(R_n-r))e^{(R_n-r)-y} - 1.
\end{align*}
Since~$(R_n-r) - y \geq -r_0$, we finally get
\begin{align*}
c_{n,r}(y)
&\geq (K(y-(R_n-r)) - 1) \wedge 0.
\end{align*}
On the other hand, for~$r_0$ and~$n$ large enough, approximating the hyperbolic terms by exponentials gives, for~$y \in (R_n-r,R_n-r+r_0)$,
\begin{align*}
c_{n,r}(y) 
&\leq 1 - K.
\end{align*}
Thus, for~$y \in (R_n-r,R_n-r+r_0)$,
\begin{align*}
\frac{1}{\sqrt{1-c_{n,r}(y)^2}} \leq \frac{1}{\sqrt{1-(1-K)^2}} \vee \frac{1}{\sqrt{1-((K(y-(R_n-r)) - 1) \wedge 0)^2}}.
\end{align*}
It follows that the function~$\frac{1}{\sqrt{1-c_{n,r}(y)^2}}$ is integrable over the interval~$(R_n-r,R_n-r+r_0)$ and the integral is bounded by a constant that does not depend on~$n$. From this and the approximation of~$f_{n,r}(y)$ given in (\ref{approx:fnr_with_c_gen}) (which also holds for~$y \in [R_n-r,R_n-r+r_0]$), we get
\begin{align*}
\int_{R_n-r}^{R_n-r+r_0} f_{n,r}(y) \rho_n(y) dy = O(\rho_n(R_n-r+r_0)) = O(e^{-\al r}).
\end{align*}
Which concludes the proof.
\end{proof}

\textbf{Acknowledgements.} The author wishes to thank Pierre Calka for suggesting the topic and providing inspiring support. Many ideas presented in this paper would not have been developed as thoroughly without his suggestions and encouragement. The author also thanks
the \emph{Laboratoire de Math\'ematiques Raphaël Salem} (\emph{Universit\'e de Rouen, France}) for its hospitality during the early stages of the project. The topic fits into the objectives of the French grant \emph{GrHyDy (Dynamic Hyperbolic Graphs) ANR-20-CE40-0002}.

\bibliographystyle{amsplain}
\bibliography{biblio.bib}

\end{document}